\documentclass[oneside,a4paper]{amsart}

\usepackage{amsthm, amsmath, amssymb, amscd,  amsfonts, amstext,
  amsbsy, mathrsfs, enumerate}

\newcommand{\pow}{\ensuremath{\mathscr{P}}}
\newcommand{\RR}{\ensuremath{\mathbb{R}}}
\def\leng{{\rm lh}}
\newcommand{\Bai}{\ensuremath{{}^\omega \omega}}
\newcommand{\Can}{\ensuremath{{}^\omega 2}}
\newcommand{\seqo}{\ensuremath{{}^{<\omega}\omega}}
\newcommand{\AD}{\ensuremath{{\rm \mathsf{AD}}}}
\newcommand{\ZF}{\ensuremath{{\rm \mathsf{ZF}}}}
\newcommand{\DCR}{\ensuremath{{\rm \mathsf{DC}(\mathbb{R})}}}
\newcommand{\ACOR}{\ensuremath{{\rm \mathsf{AC}_\omega(\mathbb{R})}}}
\newcommand{\SLOL}{\ensuremath{{\rm \mathsf{SLO^L}}}}
\newcommand{\SLO}{\ensuremath{{\rm \mathsf{SLO}}}}
\newcommand{\NFS}{\ensuremath{{\rm \lnot \mathsf{FS}}}}
\newcommand{\conc}{{}^\smallfrown}
\newcommand{\rank}[1]{\parallel #1 \parallel}
\newcommand{\imp}{\Rightarrow}
\newcommand{\fhi}{\varphi}
\newcommand{\bSigma}{\mathbf{\Sigma}}
\newcommand{\bPi}{\mathbf{\Pi}}
\newcommand{\bGamma}{\ensuremath{\mathbf{\Gamma}}}
\newcommand{\bDelta}{\mathbf{\Delta}}
\newcommand{\bN}{\mathbf{N}}
\newcommand{\F}{\mathcal{F}}
\newcommand{\G}{\mathcal{G}}
\newcommand{\loc}[2]{{#1}_{\lfloor {#2} \rfloor}}
\newcommand{\B}{\mathcal{B}}
\newcommand{\Lip}{\ensuremath{{\mathsf{Lip}}}}
\newcommand{\Bor}{\ensuremath{{\mathsf{Bor}}}}
\renewcommand{\L}{\ensuremath{{\mathsf{L}}}}
\newcommand{\W}{\ensuremath{{\mathsf{W}}}}
\newcommand{\restr}[2]{#1 \restriction #2}
\renewcommand{\sf}[1]{\ensuremath{{\mathsf{#1}}}}
\newcommand{\seq}[2]{\langle #1 \mid  #2 \rangle}
\newcommand{\onto}{\twoheadrightarrow}
\newcommand{\id}{{\rm id}}

\newtheorem{theorem}{Theorem}[section]
\newtheorem{lemma}[theorem]{Lemma}
\newtheorem*{lemmanonumber}{Lemma}
\newtheorem{corollary}[theorem]{Corollary}
\newtheorem{proposition}[theorem]{Proposition}
\theoremstyle{definition}
\newtheorem{claim}{Claim}[theorem]
\newtheorem{defin}{Definition}
\theoremstyle{remark}
\newtheorem{remark}[theorem]{Remark}

\begin{document}

\title{Baire reductions and
  good
  Borel reducibilities}
\author{Luca Motto Ros}
\date{\today}
\address{Kurt G\"odel Research Center for Mathematical Logic\\
  University of Vienna \\ W\"ahringer Stra{\ss}e 25 \\ 
 A-1090 Vienna \\
Austria}
\email{luca.mottoros@libero.it}
\thanks{Research supported by FWF (Austrian Research Fund) Grant P 19898-N18.}

\begin{abstract}
  In \cite{mottorosborelamenability} we have considered a wide class
  of ``well-behaved'' reducibilities for sets of reals. 
  In this paper we continue with the study of Borel
  reducibilities by 
  proving a dichotomy theorem for the degree-structures induced by
  good Borel reducibilities.  
  This extends and improves the
  results of \cite{mottorosborelamenability} allowing to deal with a
  larger class of notions of reduction 
  (including, among others, the Baire class $\xi$ functions).
\end{abstract}

\maketitle

\section{Introduction}

A \emph{reducibility} for  sets of reals\footnote{As usual, 
we will always identify $\RR$ with the Baire space
  $\Bai$ and call its elements ``reals''.} is simply a collection 
$\F$ of functions from $\RR$ to $\RR$ which is used to reduce a set of
reals to another 
one: given $A,B \subseteq \RR$, we say that $A$ is
\emph{$\F$-reducible} to $B$ just in case there is some $f \in \F$ such that $x
\in A \iff f(x) \in B$ for every $x \in \RR$. Such an $\F$ allows to measure
the ``relative complexity'' of the sets of reals, and
$\F$ itself can be viewed as the ``unit of measurement'' that we are using:
in general, the ``smaller'' is our set $\F$, the more accurate is our
measurement (i.e.\ the finer is the hierarchy of degrees induced by
$\F$).

The first two reducibilities that one encounters in the
literature are the collection $\W$ of all  continuous functions  and the
collection $\L$ of all  Lipschitz functions
with constant less than or equal to $1$. The
corresponding degree-structures were extensively studied (assuming
$\AD$, the \emph{Axiom of Determinacy}) by Wadge,
Steel, Van Wesep and many other set theorists (Martin, Kechris,
Louveau, Saint-Raymond to name a few), and have had many applications in
Set
Theory and Theoretical Computer Science (see for example
\cite{DCindipendentAD} or \cite{duparc}).
Some years ago, Andretta and Martin considered the collection $\Bor$ of
all Borel functions and the
collection $\sf{D}_2$ of all $\bDelta^0_2$-functions, and they proved that 
in both cases the degree-structures induced by those reducibilities look like
the Wadge one, i.e.\ like the one induced by continuous functions. In
\cite{mottorosborelamenability} we have described a general method to
extend this analysis  
to the so-called
\emph{Borel-amenable} reducibilities, among which there are e.g.\ the
continuous 
functions, the Borel functions, the collection
$\sf{D}_\xi$ of
all $\bDelta^0_\xi$-functions for $\xi<\omega_1$, i.e.\ the collection
of those $f$ such that 
$f^{-1}(D) \in \bDelta^0_\xi$ for every $D \in \bDelta^0_\xi$, and so on.
As for the previous cases, we have obtained that whenever $\F$
is a Borel-amenable set of reductions the degree-structure induced by
$\F$ looks like the Wadge one.

Since the $\sf{D}_\xi$'s form a natural stratification of the Borel functions,
the present work was mainly motivated by the natural idea
of considering the other classical stratification of the Borel
functions, namely the \emph{Baire class} $\xi$ functions. Note that
although in
\cite{mottorosnewcharacterization} it has been pointed out that there is a link
between the two stratifications, from the point of view of reducibilities
between sets of reals they clearly have a very different behaviour: in
fact, we will 
prove that the Baire functions, contrarily to the case of the $\sf{D}_\xi$'s, 
induce a degree-structure which looks like the structure of the
$\L$-degrees. This result is obtained comparing again the Baire 
stratification with the Delta stratification, and showing that the
first one gives  reducibilities which are \emph{equivalent} to (i.e.\
induce the same degree-structures as) the ones obtained glueing
together \emph{chains} of Borel-amenable sets of reductions.
On the way of studying these classes of functions, we will introduce the notion
 of \emph{good Borel reducibility} which
considerably extends the definition of Borel-amenability given in
\cite{mottorosborelamenability} (in fact it
includes, among others, all the examples quoted in this introduction):
building on our previous results, we will give a new general method to
study these 
reducibilities, which will lead to the following dichotomy
theorem.
\begin{theorem}\label{theordic1}
  Assume $\AD+\DCR$. If $\F$ is a good Borel set of reductions
  then it induces either a Lipschitz-like or a Wadge-like hierarchy of degrees.
\end{theorem}

This improves many of the results obtained in
\cite{mottorosborelamenability} and is a first step toward proving the na\"ive
conjecture that the dichotomy above should hold for all 
``reasonable'' Borel sets of reductions. 
\newline

The paper is organized as follows: in Section \ref{sectionpreliminaries}
we will fix some notation and review some of the results about $\L$-degrees
and Borel-amenable reducibilities that will be needed for the rest
of the work. (We will systematically omit the proofs of these results
--- the reader interested in 
some of them can consult \cite{wadgethesis},
\cite{vanwesepwadgedegrees} or the more succinct \cite{andrettaslo} for 
Lipschitz degrees, and \cite{mottorosborelamenability} for 
Borel-amenable sets of reductions.) In Section \ref{sectiondichotomy}
we will give the 
definition of good Borel reducibility, while in Section
\ref{sectionSDP} we will introduce the Strong Decomposition Property
and  prove some
results which essentially form the framework of the proof of our
dichotomy theorem.  In Section \ref{sectionLipandchains} we will deal
with the cases 
of Lipschitz functions and chains of reductions,
 and these results
will in turn be used in Section \ref{sectionBaire} to analyze the hierarchies
of degrees induced by Baire functions. Finally, in Section
\ref{sectioncomparing} we will show how to compare different
degrees-structures (in particular showing how to obtain a certain
hierarchy from the finer ones).

\section{Preliminaries}\label{sectionpreliminaries}

Unless otherwise stated, we will always assume
$\ZF+\SLOL+\NFS+\DCR$ (see
\cite{mottorosborelamenability} 
and \cite{andrettaslo} for the definitions and for
a brief account
on these axioms). Anyway $\SLOL$ and $\NFS$ are easy consequences of
$\AD$, thus one can also  
safely work in the most well-known theory $\ZF+\AD+\DCR$. In both cases, 
all the ``determinacy axioms'' are
used in a
local way throughout the paper, thus e.g.\ to compare
Borel sets is enough to assume Borel-determinacy.
Our notation and terminology is quite standard, and
 we systematically refer the reader to \cite{mottorosborelamenability}
 for the basic definitions.  
We just recall here that ${}^A B$ denotes the set of all functions from $A$
into $B$, that a set $F \subseteq \Can$ is said \emph{flip-set}
whenever $\exists ! n (z(n) \neq w(n)) \imp ({z \in F} \iff {w \notin F})$
for every $z,w \in \Can$, and that, given any pointclass $\Gamma
\subseteq \pow(\RR)$, a 
function $f \colon \RR \to \RR$ is said \emph{$\Gamma$-function} if
$f^{-1}(D) \in \Gamma$ for every $D \in \Gamma$.\\

Let now $\F \subseteq
{}^\RR \RR$ be a family of functions which is closed under
composition, contains $\L$ and admits a surjection $j \colon \RR \onto
\F$, that is 
a so-called \emph{set of reductions}. Recall that $A \leq_\F B \iff A
= f^{-1}(B)$ for some $f \in \F$ 
(notice that $A \leq_\F B \iff \neg A \leq_\F \neg B$), and let $<_\F$
be the strict 
relation associated to $\leq_\F$.
 Since $\leq_\F$ is a preorder, we
can canonically define the equivalence relation $\equiv_\F$ 
and study the partial order $\leq$ induced by $\leq_\F$ on the
equivalence classes of $\equiv_\F$, which are  called \emph{$\F$-degrees}.
A set $A$ (or its $\F$-degree $[A]_\F$) is said to be
\emph{$\F$-selfdual} if and only if 
$A \leq_\F \neg A$ (otherwise it is
\emph{$\F$-nonselfdual}), and $\{[A]_\F,[\neg A]_\F\}$ is called
\emph{nonselfdual pair} whenever $A \nleq_\F \neg A$.
The \emph{Semi-Linear Ordering Principle} for $\F$ is the statement
\begin{equation}
\tag{$\SLO^\F$}
\forall A,B \subseteq \RR
({A \leq_\F B} \vee
  {\neg B \leq_\F A}).
\end{equation}
\noindent
Under $\SLO^\F$ we have that if $A$ and $B$ are $\leq_\F$-incomparable,
then $B \equiv_\F \neg A$:
thus the ordering induced on the $\F$-degrees is \emph{almost} a
linear-order (it becomes indeed linear if each degree is identified
with its dual). 
If now $\F \subseteq \G \subseteq {}^\RR \RR$ are sets of
  reductions, then  $\leq_\G$ is clearly coarser than $\leq_\F$: hence
  $A \leq_\F B \imp 
A \leq_\G B$, if $A$ is $\F$-selfdual then it is also $\G$-selfdual,
and $[A]_\F \subseteq [A]_\G$.
Moreover the following basic lemma holds.

\begin{lemma}[$\ZF$] \label{lemmabasic}
Let $\F \subseteq \G \subseteq {}^{\RR}\RR$ be two sets of reductions. Then
$\SLO^\F \imp \SLO^\G$, and assuming $\SLO^\F$ we have 
$\forall A,B \subseteq \RR ({A<_\G B} \imp {A <_\F B})$.
\end{lemma}

This lemma will be mostly used when $\F = \L$: this in particular means
that under our axiomatization we also have $\SLO^\F$ for \emph{every}
set of reductions 
$\F$. Moreover it easily implies that $\leq_\F$ is well-founded (since
$\leq_\L$ is): therefore we can associate a rank $\rank{\cdot}_\F$ to
each set $A \subseteq \RR$ (resp.\ $\F$-degree $[A]_\F$), and speak of
\emph{successor} and \emph{limit} sets (resp.\ $\F$-degrees), and of the
\emph{cofinality} of a set (resp.\ of an $\F$-degree) with the obvious
meaning.
The next theorem sum up  the general properties
of sets of reductions --- see Theorem 3.1 in
\cite{mottorosborelamenability}. Recall that given $A,B,A_n \in \RR$,
$\bigoplus_n A_n$ denotes the set $\bigcup_n (n \conc A_n)$, while $A
\oplus B$ denotes $\bigoplus_n C_n$, where $C_{2k} = A$ and $C_{2k+1} = B$ for
every $k \in \omega$.

\begin{theorem} \label{theorgeneralproperties}
Let $\F \subseteq {}^\RR \RR$ be a set of reductions. Then
\begin{enumerate}[i)]
\item $\leng(\leq_\F)= \Theta$, where $\Theta = \sup \{\alpha \mid f
  \colon \RR \onto \alpha \text{ for some surjection }f\}$; 
\item anti-chains have size at most $2$ and are of the form
  $\{[A]_\F,[\neg A]_\F\}$ for some set $A$;
\item $\RR \nleq_\F \neg \RR = \emptyset$ and if $A \neq \emptyset,
  \RR$ then $\emptyset, \RR <_\F A$;
\item if $A \nleq_\F \neg A$ then $A \oplus \neg A$ is $\F$-selfdual
  and is the successor of both $A$ and $\neg A$. In particular, after
  an $\F$-nonselfdual pair 
  there is a single $\F$-selfdual degree;
\item if $A_0 <_\F A_1 <_\F \dotsc$ is a countable $\F$-chain of subsets of
  $\RR$ then $\bigoplus_n A_n$ is $\F$-selfdual and is
  the supremum of these
  sets. In particular if $[A]_\F$ is
  limit of countable cofinality then $A \leq_\F \neg A$;
\item if $A \nleq_\F \neg A$ and $\G \subseteq \F$ is another set of reductions
 then $[A]_\F = [A]_\G$. In particular, $[A]_\F = [A]_\L$.
\end{enumerate}
\end{theorem}

Thus to determine the hierarchy of degrees induced by some
$\F$ we have only to understand what happens after a single selfdual
degree and at limit levels of uncountable cofinality. 

Given any set of reductions $\F$ (or even just any \emph{set of
  functions}) we can define its \emph{characteristic set} 
$\Delta_\F = \{A \subseteq \RR \mid A \leq_\F \bN_{\langle 0
  \rangle}\}$,
which is formed by all sets $A \subseteq \RR$ which are
\emph{simple}
from the ``point of view'' of $\F$.
As a simple excercise one can check that $\Delta_{\sf{D}_\xi} = \bDelta^0_\xi$
for every countable $\xi$, and that $\Delta_\Bor = \bDelta^1_1$.
It is easy to see that if $\F$ is closed under composition then every
$f \in \F$ 
is a $\Delta_\F$-function (even if the converse is not always true --- see
\cite{mottorosborelamenability} for a counter-example), thus it make sense
to define the \emph{saturation} of $\F$
\[ {\rm Sat}(\F) = \{ f \in {}^\RR \RR \mid f \text{ is a }
\Delta_\F\text{-function}\},\]
and to say that $\F$ is \emph{saturated} just in case $\F = {\rm Sat}(\F)$.
Moreover it is easy to see that $\F \subseteq \G$ implies $\Delta_\F
\subseteq \Delta_\G$ 
(the converse is not true in general, unless $\G = \Bor$ or $\G =
\sf{D}_\xi$ for some 
$\xi < \omega_1$, and $\bN_s \in \Delta_\F$ for every $s \in \seqo$).
Finally, $\F$ is said to be \emph{Borel} if
$\{\bN_s \mid s \in \seqo\} \subseteq\Delta_\F
\subseteq\bDelta^1_1$,
that is if $\F \subseteq \Bor$ and $\F$ recognizes as simple all the
basic clopen sets of $\RR$. 

Here are some basic facts about $\L$-degrees: after a selfdual $\L$-degree
there is always another selfdual $\L$-degree, and a limit 
$\L$-degree is selfdual if and only if it is of countable
cofinality, otherwise it is nonselfdual.
Thus after a selfdual $\L$-degree
$[A]_\L$ there is always an
$\omega_1$-chain of consecutive selfdual
$\L$-degrees, and therefore the $\L$-hierarchy looks like this:
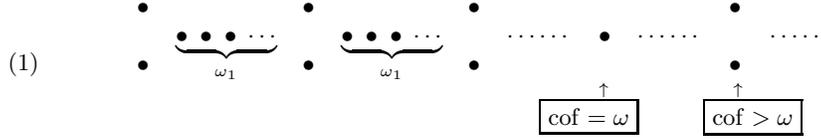
\begin{small}
\begin{equation}\label{pictureLipschitz}\index{Hierarchy!Lipschitz}
\begin{array}{llllllllll}
\bullet & & \bullet & & \bullet & & & &\bullet
\\
& \smash[b]{\underbrace{\bullet \; \bullet \;
\bullet \; \cdots}_{\omega_1}} & &
\smash[b]{\underbrace{\bullet \; \bullet \; \bullet \;
\cdots}_{\omega_1}} & & \cdots\cdots &
\bullet & \cdots\cdots & & \cdots\cdots
\\
\bullet & & \bullet & & \bullet & & & & \bullet
\\
& & & & & &
\stackrel{\uparrow}{\makebox[0pt][r]{\framebox{${\rm cof} = \omega$}}} & &
\, \stackrel{\uparrow}{\makebox[0pt][l]{\framebox{${\rm cof} > \omega$}}}
\end{array}
\end{equation}
\end{small}
\noindent
Any structure of this kind will be called \emph{Lipschitz-like}.

Now we turn our attention to Borel-amenable sets of reductions.

\begin{defin}
  A set of reductions $\F$ is \emph{Borel-amenable} if:
  \begin{enumerate}[i)]
  \item $\Lip \subseteq \F \subseteq \Bor$;
  \item for every $\Delta_\F$-partition $\seq{D_n}{n \in \omega}$ and every collection
$\{f_n \mid n \in \omega\} \subseteq \F$ we have that
\[ f = \bigcup\nolimits_{n \in \omega} (\restr{f_n}{D_n}) \in \F,\]
where $\Lip$ is the collection of \emph{all} Lipschitz functions
(with any constant).
\end{enumerate}
\end{defin}

As an example of Borel-amenable reducibility one can take any of
the $\sf{D}_\xi$'s or $\Bor$. It turns out that for every Borel-amenable
set of reductions $\F$ there is some $\xi \leq \omega_1$ (called the
\emph{level} of $\F$) such that $\Delta_\F = \bDelta^0_\xi$ (where,
 with a little abuse
of notation, we put $\bDelta^0_{\omega_1} = \bDelta^1_1$): thus, in particular,
${\rm Sat}(\F)$ is always either one of the $\sf{D}_\xi$'s  or $\Bor$.
Moreover for any Borel-amenable set of reductions $\F$ we have
the following lemma.

\begin{lemma}\label{lemmapartition}
  Let $\seq{D_n}{n \in \omega}$ be a $\Delta_\F$-partition of $\RR$ and let
$A \neq \RR$.
  \begin{enumerate}[a)]
  \item $\forall n\in \omega(A \cap D_n \leq_\F A)$.
  \item If $C \subseteq \RR$ and $A \cap D_n \leq_\F C$ for every $n
    \in \omega $ then
$A \leq_\F C$.
\item If $\forall n \in \omega(A\cap D_n <_\F A)$ then $A \leq_\F \neg
  A$. Moreover, 
if $D_n = \emptyset$ for all but finitely many $n$'s then $A$ is not limit.
  \end{enumerate}
\end{lemma}

Let us say that $\F$ has the \emph{Decomposition Property} (\textbf{DP}
for short) if for every selfdual $A \notin \Delta_\F$ there is a
$\Delta_\F$-partition $\seq{D_n}{n \in \omega}$ of $\RR$ such that $A \cap D_n <_\F A$
for every $n$ or, equivalently, if for every selfdual $A$ which is
$\L$-minimal
in $[A]_\F$ one has that $A \leq_\L \neg A$ (and $A$ is either limit
or successor of a nonselfdual pair with respect to $\leq_\L$). 
Using this property (which turns out to be a consequence of 
Borel-amenability) we have proved in \cite{mottorosborelamenability} 
that both after a single selfdual degree
and at limit levels of uncountable cofinality there is a nonselfdual pair.
Thus the hierarchy of
degrees induced by $\F$ looks like this:

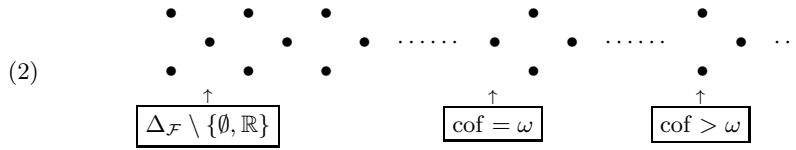
\begin{small}
\begin{equation}
\begin{array}{llllllllllllll}
\bullet & & \bullet & & \bullet & & & & \bullet & & & \bullet
\\
& \bullet & & \bullet & & \bullet & \cdots \cdots & \bullet
& & \bullet &\cdots\cdots & & \bullet & \cdots
\\
\bullet & & \bullet & & \bullet & & & & \bullet & & & \bullet
\\
& \, \makebox[0pt]{$\stackrel{\uparrow}{\framebox{$\Delta_\F \setminus
\{\emptyset,\RR\}$}}$}& & & & & &
\, \makebox[0pt]{$\stackrel{\uparrow}{\framebox{${\rm cof} = \omega$}}$}
& & & &
\, \makebox[0pt]{$\stackrel{\uparrow}{\framebox{${\rm cof} > \omega$}}$}
\end{array}
\end{equation}
\end{small}
\noindent Any structure of this kind 
will be called \emph{Wadge-like}.
Notice that one gets both the Wadge
hierarchy and the degree-structures induced by $\Bor$ and $\sf{D}_2$
as particular instances of the previous result. 

The \textbf{DP}
allows also to compare different sets of reductions by means of
the degree-structures induced by them. Let us say that two sets of reductions
$\F$ and $\G$ are \emph{equivalent} ($\F \simeq \G$ in symbols) if they induce
the same hierarchy of degrees, that is if for every $A,B \subseteq \RR$ we have
$A \leq_\F B \iff A \leq_\G B$:
if $\F$ and $\G$ are 
 Borel-amenable sets of reductions then\label{equivalentreductions}
\[ \F \simeq \G \iff \Delta_\F = \Delta_\G,\]
that is $\F \simeq \G$ if and only if their degree-structures 
coincide on the first nontrivial level. In particular, since $\Delta_\F =
\Delta_{{\rm Sat}(\F)}$, we have that $\F \simeq {\rm Sat}(\F)$.

We want to conclude this section by recalling how to construct, given
an $\F$-selfdual set $A \subseteq \RR$, where $\F$ is of level $\xi <
\omega_1$, its successor degree(s). Fix
an increasing sequence of ordinals $\langle \mu_n\mid n \in \omega
\rangle$ cofinal in $\xi$, and a sequence of sets $P_n$
such that $P_n \in \bPi^0_{\mu_n} \setminus \bSigma^0_{\mu_n}$.
Let $\langle \cdot,\cdot \rangle\colon \omega \times \omega \to \omega$ be
any  bijection,
and for any zero-dimensional
space\footnote{When $\mathscr{X}=\omega$ we will simply drop the symbol $\mathscr{X}$ in all the
relevant notation.} $\mathscr{X}$ define the homeomorphism
\[ \bigotimes\nolimits^\mathscr{X} \colon  {}^\omega ({}^\omega \mathscr{X})
\to  {}^\omega \mathscr{X} \colon 
\langle x_n\mid  n \in 
\omega \rangle \mapsto x=\bigotimes^\mathscr{X}\nolimits_n x_n,\]
by letting $x(\langle n,m \rangle) = x_n(m)$, and, conversely, the
``projections'' 
$\pi^\mathscr{X}_n\colon  {}^\omega \mathscr{X} \to {}^\omega \mathscr{X}$
by setting  $\pi^\mathscr{X}_n(x) = \langle
x(\langle n,m \rangle)\mid m \in \omega \rangle$ (clearly, every ``projection''
 is
surjective, continuous and open). Note that given  a sequence of functions
$f_n\colon 
{}^\omega \mathscr{X} \to {}^\omega \mathscr{X}$, we can use the homeomorphism
$\bigotimes\nolimits^\mathscr{X}$ to  define
the function
\[ \bigotimes\nolimits^\mathscr{X}\langle f_n\mid
n\in\omega\rangle=\bigotimes\nolimits^\mathscr{X}_n f_n\colon 
{}^\omega \mathscr{X} \to {}^\omega \mathscr{X} \colon  x \mapsto
\bigotimes\nolimits^\mathscr{X}_n f_n(x),\]
and it is not hard to check that $\bigotimes\nolimits^\mathscr{X}_n f_n$
is continuous  if and
only  if all
the $f_n$'s are continuous. 
Now
consider the sets
\[ \Sigma^\xi(A) = \{x \in \RR\mid \exists n (\pi_{2n}(x) \in P_n
\wedge \forall i<n(\pi_{2i}(x) \notin P_i) \wedge \pi_{2n+1}(x) \in
A)\}\]
and
\[ \Pi^\xi(A) = \Sigma^\xi(A) \cup R_\xi,\]
where $R_\xi = \{x\in \RR\mid \forall n(\pi_{2n}(x) \notin P_n) \}$:
 it turns
out
that for every $A \leq_\F \neg A$ the sets
$\Sigma^\xi(A)$ and $\Pi^\xi(A)$ are $\leq_\F$-incomparable and  are
 the immediate successors of $A$.

\section{Good Borel
  reducibilities}\label{sectiondichotomy}

In \cite{mottorosborelamenability} we have studied a special kind of 
Borel reducibilities but, as we will see later in this paper,
there are also other ``natural'' sets of reductions which are not of this
kind (namely Lipschitz functions, uniformly continuous functions, Baire
functions, and so on). Thus our goal is to weaken the condition
of Borel-amenability in order to be able to study also these
other examples.
Recall that the second condition of Borel-amenability says that $f =
\bigcup_n (\restr{f_n}{D_n})\in \F$ 
whenever $\{f_n \mid n \in \omega\} \subseteq \F$ and $\seq{D_n}{n \in
  \omega}$ is a 
$\Delta_\F$-partition of $\RR$. We will weaken this
condition both 
allowing to use only $f_n$'s which are in $\L$ and using the concept
of \emph{boundness} (in a pointclass):
a pointclass $\Lambda$ is
  \emph{($\L$-)bounded} in an $\L$-pointclass $\Gamma$ if there is some $A \in \Gamma$
  such that $B \leq_\L A$ for every $B \in \Lambda$ (which  in
  particular implies that $\Lambda \subseteq \Gamma$).
Moreover, $\seq{D_n}{n \in \omega}$ is a \emph{$\Gamma$-bounded partition} of
$\RR$ if it is a $\Gamma$-partition of $\RR$ such that $\{ D_n \mid
n \in \omega\}$ is bounded in $\Gamma$.

\begin{defin}
We say that $\F$ satisfies the \emph{partitioning condition}
(\textbf{PC} for short) if for every $\Delta_\F$-\emph{bounded} partition
$\seq{D_n}{n \in \omega}$ of $\RR$ and every collection $\{ f_n\mid
n \in \omega\}\subseteq \L$  one has that
\[ f = \bigcup\nolimits_{n \in \omega} (\restr{f_n}{D_n})\in \F.\]
\end{defin}

\noindent
Notice that there are just three types of Borel reducibilities that can satisfy the \textbf{PC}: 

\begin{description}
 \item[TYPE I] $\Delta_\F = \{ A \subseteq \RR \mid A \leq_\L \bN_s \text{ for some }s \in \seqo \}$;
 \item[TYPE II] $\Delta_\F = \bDelta^0_\xi$ for some countable $\xi$ or $\Delta_\F = \bDelta^1_1$;
 \item[TYPE III] $\Delta_\F = \bDelta^0_{<\lambda} = \bigcup_{\mu < \lambda} \bDelta^0_\mu$ for some countable limit ordinal $\lambda$.
\end{description}

The proof of
this fact is 
essentially the same of Proposition 4.3 in
\cite{mottorosborelamenability}. By Borel determinacy we need to consider just two cases\footnote{The principle $\SLO^\L$ for Borel sets, which follows from Borel determinacy, implies that for every pair of $\L$-pointclasses $\Gamma,\Lambda \subseteq \bDelta^1_1$ either $\Gamma \subseteq \Lambda$ or $\breve{\Lambda} \subseteq \Gamma$: therefore if both $\Gamma$ and $\Lambda$ are selfdual either $\Gamma \subseteq \Lambda$ or $\Lambda \subseteq \Gamma$.}, namely  $\bDelta^0_1 \subsetneq \Delta_\F$ and $\Delta_\F \subseteq \bDelta^0_1$. In the first case, assume that $\Delta_\F \neq\bDelta^1_1$ and that $\F$ is not of type III, and let $1 < \xi
<\omega_1$ be the smallest ordinal such that $\Delta_\F \subseteq
\bDelta^0_\xi$. If $D \in \bDelta^0_\xi$, then there is
  some partition $\langle D_n \mid n \in \omega \rangle$ of $\RR$ such that $D =
  \bigcup_{i \in I} D_i$ for some $I \subseteq \omega$ and $D_n \in
  \bPi^0_{\mu_n}$ for some $\mu_n < \xi$ (see Theorem 4.2 in 
\cite{mottorosnewcharacterization}). Since $\bDelta^0_\mu
  \subsetneq \Delta_\F$ for every $\mu < \xi$ (by minimality of
  $\xi$), we have that $\bigcup_{\mu < \xi} \bPi^0_\mu \subseteq \Delta_\F$ by Borel determinacy again, and hence that $\{D_n \mid n\in \omega\} \subseteq \bigcup_{\mu<\xi} \bPi^0_\mu$
  is bounded in $\Delta_\F$ (when $\xi$ is limit use the fact that $\F$ is not of type III). Let $g_0,g_1$ be the constant functions with value $\vec{0}$ and $\vec{1}$, respectively, 
  and put
  $f_i = g_0$ if $i \in I$ and $f_i=g_1$ otherwise. By
  the \textbf{PC},
$f= \bigcup_{n \in \omega}(\restr{f_n}{D_n}) \in \F$ and $f^{-1}(\bN_{\langle 0 \rangle}) = D$, 
i.e.\ $D \in \Delta_\F$: therefore $\bDelta^0_\xi \subseteq \Delta_\F$ and $\F$ is of type II. Finally, the argument for the case $\Delta_\F \subseteq \bDelta^0_1$ is similar to the previous one (it suffices to prove that if $\F$ is not of type I then $\Delta_\F = \bDelta^0_1$), and it is left to the reader.

As a corollary,
one gets  that  $\Delta_\F$ is an
algebra of sets (i.e.\ it is 
closed under complementation and finite intersections).
Moreover, the \textbf{PC} implies $\F \supseteq \L$: therefore, since
$\Delta_\F \subseteq \bDelta^1_1$ already implies that there is a
surjection $j \colon \RR \onto \F$, a Borel \emph{set 
of functions} $\F$ which satisfies the \textbf{PC} is also a \emph{set
of reductions} just in case it is closed under composition. 

Another consequence of the \textbf{PC} is Lemma 4.4 of
\cite{mottorosborelamenability} (since any \emph{finite}
$\Delta_\F$-partition of $\RR$ is obviously bounded in
$\Delta_\F$): if $D \subseteq D'$ are in 
$\Delta_\F$ and $A \subseteq \RR$ is such that $A \cap D' \neq \RR$
then $A \cap D \leq_\F A \cap D'$ (in particular, if $A \neq \RR$ then
$A \cap D \leq_\F A$ for every $D \in \Delta_\F$). Finally, the
\textbf{PC} allows to reprove Lemma \ref{lemmapartition} (using almost
the same argument) 
in  case  $\F$ is a Borel set of reductions which
satisfies the \textbf{PC}
(but non necessarily a Borel-amenable one) as soon as the partition  $\seq{D_n}{n \in \omega}$ is bounded in $\Delta_\F$ and
part  b) is replaced by the following condition:
\begin{equation}\label{eqstar}
\tag{$\star$} \text{if } C \subseteq \RR \text{ and }A \cap D_n
\leq_\L C \text{ for every }n \text{ then } A \leq_\F C. 
\end{equation}

Besides the \textbf{PC},
there is also another condition which is somewhat hidden in the
definition of Borel-amenability.

\begin{defin}\label{defsigmabounded}
If $\Delta$ is an $\L$-pointclass, we
say that an arbitrary function $f \colon \RR \to \RR$ is
\emph{$\sigma$-bounded (in $\Delta$)} if for every 
countable collection $\{ D_n \mid n \in \omega \}$ bounded in 
  $\Delta$  one has that $\{
  f^{-1}(D_n)\mid n \in \omega \}$ is bounded in $\Delta$ as well (thus, in
  particular, $f$ is a $\Delta$-function).
A set of functions  $\F$ satisfies the
  \emph{$\sigma$-boundness condition} (\textbf{$\sigma$-BC} for short)
  if every $f\in \F$ is $\sigma$-bounded in $\Delta_\F$.
\end{defin}

It is easy to check that if $\F$ is of type II then \emph{every}
countable collection $\{ D_n \mid n  \in \omega\} \subseteq\Delta_\F$
is bounded 
in $\Delta_\F$, thus \textbf{$\sigma$-BC} becomes relevant only when
$\F$ is not of type II (this is the reason for which this condition
was not explicity highlighted in \cite{mottorosborelamenability}).
On the other hand, if $\F$ is of type I or III the
\textbf{$\sigma$-BC} turns out to be equivalent to the seemingly
stronger statement ``if $f \in \F$ and $\Gamma$ is bounded in
$\Delta_\F$ (with $\Gamma$ of arbitrary size) then $\{f^{-1}(C) \mid C
\in \Gamma \}$ is bounded in $\Delta_\F$'': this is because in the
cases under consideration $\Delta_\F$ has ``countable cofinality''
(i.e.\ there is a countable chain which is $\L$-unbounded in
$\Delta_\F$), and therefore from every pointclass $\L$-unbounded in
$\Delta_\F$ one can extract a countable subpointclass which is still
$\L$-unbounded in $\Delta_\F$. 

We will see  that the \textbf{PC} and the \textbf{$\sigma$-BC} are
strong enough to 
civilize the hierarchy of degrees induced by $\F$, so let us give the
following definition:

\begin{defin}
  A Borel set of reductions is said to be a \emph{good Borel
    reducibility} if it 
  satisfies both the \textbf{PC} and the \textbf{$\sigma$-BC}. 
  The collection of all good Borel reducibilities
  will be denoted by $\sf{GR}$. 
\end{defin}

Note that the Borel-amenable reducibilities form a \emph{proper}
subset of good Borel  
reducibilities of type II, 
as $\Lip \nsubseteq \sf{D}_\xi^\L$ for any $\xi <
\omega_1$, where $\sf{D}^\L_\xi$ is the collection of those $f$ which
are in $\L$ on a (countable)  $\bDelta^0_\xi$-partition (in particular
this proves that each  
$\sf{D}_\xi^\L$ does not contain \emph{any} Borel-amenable set of
reductions).  In fact, one can easily check that the pseudoidentity
$\id^- \colon \RR \to \RR \colon x \mapsto \seq{x(n+1)}{n \in \omega}$
is such that for \emph{every} countable partition $\seq{D_n}{n \in \omega}$ and
\emph{every} family $\{ f_n \mid n \in \omega \}\subseteq \L$ there is some
$n_0$ such that $\restr{f_{n_0}}{D_{n_0}}\neq
\restr{\id^-}{D_{n_0}}$ (the argument is based on the Baire Category
Theorem and is almost identical to the one used in Remark 6.2 of
\cite{mottorosborelamenability}).

\section{The Strong Decomposition 
Property and the Dichotomy 
  Theorem}\label{sectionSDP}

First we want to prove that every good Borel reducibility
$\F$ has the following bounded version of the Decomposition Property.

\begin{defin}
A set $A\subseteq \RR$ has the \emph{Strong Decomposition Property}
with respect to a Borel set of reductions $\F$ if there is a
$\Delta_\F$-\emph{bounded} partition
  $\seq{D_n}{n\in \omega}$ of $\RR$ such that $A \cap D_n <_\F A$ for
  every $n$. 

  A Borel set of reductions has the \emph{Strong Decomposition
    Property} ({\rm \textbf{SDP}} for short) if  every $A\subseteq
  \RR$ such that $A \leq_\F
  \neg A$ and $A \notin \Delta_\F$ has the Strong Decomposition
  Property with respect to $\F$.
\end{defin}

Observe that if $\F$ is of type II then the \textbf{SDP} is equivalent to
the \textbf{DP} by the observation following Definition
\ref{defsigmabounded}. 

\begin{remark}\label{remSDP}
If $\F$
satisfies the \textbf{$\sigma$-BC}, then $A$ has the Strong
Decomposition Property with
respect to $\F$ if and only if there is some  $B$ in
$[A]_\F$ \label{SDPinvariant} which  has the Strong Decomposition
Property with respect
to $\F$. In fact, let $f\in\F$ be a reduction of $A$ into $B$, and let
$\seq{D'_n}{n \in \omega}$ be a $\Delta_\F$-bounded partition of $\RR$ such
that $B \cap D'_n <_\F B$. Put $D_n = f^{-1}(D'_n)$: 
$\seq{D_n}{n \in \omega}$ is a $\Delta_\F$-bounded partition of $\RR$ by the
\textbf{$\sigma$-BC}, and since $f$ witnesses $A\cap D_n \leq_\F B \cap
D'_n$ we have also
$A \cap D_n \leq_\F B \cap D'_n <_\F B \leq_\F A$.
\end{remark}

If $\F$ is good (and satisfies a simple technical
 condition) then the
\textbf{SDP} can also be recast in an
equivalent way.

\begin{proposition}\label{propequivSDP}
  Let $\F \in \sf{GR}$ be 
  such that $k \conc B \leq_\F B$ for every $k \in \omega$ and every $B
  \subseteq \RR$ (we can require for instance that $\Lip(2) \subseteq
  \F$). Then 
  for every $A \subseteq \RR$ the following are equivalent:
  \begin{enumerate}[i)]
  \item $A$ has the Strong Decomposition Property with respect to $\F$;
  \item if $B$ is $\L$-minimal in $[A]_\F$ then $B\leq_\L \neg B$.
  \end{enumerate}
Moreover, 
    $B$ is either limit or successor of a nonselfdual pair with
    respect to $\leq_\L$.
\end{proposition}

\begin{proof}
If $A = \emptyset$ or $A = \RR$ neither \textit{i)} nor \textit{ii)}
can hold, thus 
we can assume $A \neq \emptyset,\RR$.
  If $A$ has the Strong Decomposition Property with respect
  to $\F$, let
  $\seq{D_n}{n \in \omega}$ be a $\Delta_\F$-bounded partition of $\RR$ such
  that $A \cap D_n <_\F A$ for every $n$ and put $B_n = A \cap
  D_n$. Clearly we can not have that 
  there is an $m \in \omega$ such that $B_n \leq_\L B_m$ for every
  $n\in \omega$, otherwise $A \leq_\F B_m$ by condition
  \eqref{eqstar}, a contradiction! 
Therefore $\forall m \exists n (B_n \nleq_\L B_m)$ and hence $B  =
\bigoplus_n B_n$ is 
$\L$-selfdual. Moreover $B \leq_\L C$ for every $C \in [A]_\F$ since
$B_n <_\L C$ for every $n 
\in \omega$:   on the other hand, $A \leq_\F B$ by condition
\eqref{eqstar} again, hence $B$ is $\L$-minimal in  
$[A]_\F$.

Assume now that \textit{ii)} holds. Recall that if $C$ is an arbitrary subset of $\RR$
and $k \in \omega$, then $\loc{C}{k}$ denotes the set $\{ x \in \RR
\mid k \conc x \in C \}$. It is a classical fact that since $B$ is $\L$-selfdual we have 
 $\loc{B}{k} <_\L B$ for
every $k \in \omega$, and hence by $\L$-minimality of $B$ in $[A]_\F$ we
have also $\loc{B}{k}<_\F B$. Our technical condition implies that $B
\cap\bN_{\langle k \rangle} \leq_\F \loc{B}{k}$, and since
$\seq{\bN_{\langle k \rangle}}{k\in \omega}$ is
always bounded in $\Delta_\F$ then $B$ has the Strong Decomposition
Property with
respect to $\F$. But by Remark \ref{remSDP} this implies that  $A$ has
the Strong 
Decomposition Property
with respect to $\F$ as well.
The last part of the statement easily follows from our technical
condition and the $\L$-minimality of $B$ in $[A]_\F$, as if $B$ is the successor with respect to $\leq_\L$ of an $\L$-selfdual set $B'$ then $B \equiv_\L 0 \conc B'$ (see e.g.\ \cite{vanwesepwadgedegrees} or \cite{andrettaslo} for a proof of this easy fact).
\end{proof}

We will  prove in Theorem \ref{theorSDP} that the
\textbf{$\sigma$-BC} already implies  the \textbf{SDP}
(also in absence of the \textbf{PC} and of the other technical
condition), but first we need the next proposition, which is a deep application of the Martin-Monk
method and a
further strengthening of Theorem 16 in \cite{andrettamartin} and of Theorem 5.3 in
\cite{mottorosborelamenability}.
As for the other results of this kind, we will
use the following lemma (probably due to Kuratowski, see Lemma 5.1 in
\cite{mottorosborelamenability} and the references given there).

\begin{lemma}[\ZF+\ACOR]\label{lemmarel}
Let $d$ be the usual metric on $\RR$,
$\tau$ the topology induced by $d$, and let $\xi$ be any
nonzero countable ordinal. For any family $\{D_n \mid n \in \omega\} \subseteq
\bDelta^0_\xi$ there is a metric $d'$ on $\RR$ such that
\begin{enumerate}[i)]
\item $(\RR,\tau')$ is Polish and zero-dimensional, where $\tau'$ is the
topology induced by $d'$;
\item $\tau'$ refines $\tau$;
\item each $D_n$ is $\tau'$-clopen;
\item there is a countable clopen basis $\B'$ for $\tau'$ such that
$\B' \subseteq \bDelta^0_\xi$.
\end{enumerate}
\end{lemma}

\begin{proposition}\label{propSDPforminimal}
Assume that $\F$ is of type {\rm I}, {\rm II} or {\rm III} and has the {\rm \textbf{$\sigma$-BC}}.
Let $A \subseteq \RR$ be such that $A \leq_\F \neg A$, $A \notin
\Delta_\F$ and $A$ is $\L$-minimal in its $\F$-degree. Then $A$ 
has the Strong Decomposition Property with respect to $\F$.
\end{proposition}

\begin{proof}
 We start by considering the case in which $\F$ is of type III, the
other cases  will be treated in a similar way.
First observe that since $A \leq_\F \neg A$ we have $A \nleq_\F {A \cap D} \iff {A \cap D} <_\F A$.
Let $\xi<\omega_1$ be
such that $\Delta_\F = \bDelta^0_{<\xi}$, and let $f \in \F$ be any reduction of $A$ into
$\neg A$.
Toward a contradiction, assume that for every $\Delta_\F$-bounded partition
$\seq{D_n}{n \in \omega}$ of $\RR$ there is some $n_0 \in \omega$ such that $A \leq_\F
A \cap D_{n_0}$.
We will construct three sequences\footnote{The major difference from the present argument and the proof of Theorem 5.3 in \cite{mottorosborelamenability} is that in this case we will not require that $C_{n+1} \subseteq C_n$, so that we will have to use some special $f_n$'s (which ``jump'' from one $C_n$ into the next one) rather than the identity function.}
\[
 \seq{C_n}{n \in \omega},\ \seq{d_n}{n \in \omega},\
\seq{f_n}{n \in \omega}
\]
such that for every $n \in \omega$:
\begin{enumerate}[i)]
\item $C_n \in \Delta_\F$ and $A \leq_\F A \cap C_n$;
\item $f_n \colon  \RR \to C_n$ is such that $f_n^{-1}(A \cap C_n)
= A$ (i.e.\ $f_n$ reduces $A$ to $A \cap C_n$), and hence also 
$f_n^{-1}(\neg A \cap C_n) = \neg A$;
\item $d_n$ is a metric on $\RR$ such that the induced topology $\tau_n$  is 
zero-dimensional and Polish, refines all the previous $\tau_m$'s ($m \leq n$), 
$C_n$ is $\tau_n$-clopen,
and both $f_n \colon  (\RR, \tau_{n+1}) 
\to (\RR,\tau_n)$ and $f_n \circ f \colon  (\RR, \tau_{n+1}) \to (\RR,\tau_n)$ are
continuous;
\item for every $m \leq n$ and every $x,y \in C_{n+1}$
\begin{equation}\label{eqCauchy}
\tag{$*$} d_m (g_m \circ \dotsc \circ g_n(x), g_m \circ \dotsc \circ g_n(y)) < 2^{-n}
\end{equation}
where for each $i$ either $g_i = \restr{f_i \circ f}{C_{i+1}}$ or 
$g_i = \restr{f_i}{C_{i+1}}$. 
\end{enumerate}
Observe that by ii) we have that $f_n \circ f \colon  \RR \to C_n$ is such that
\[ \forall x \in C_{n+1} (x \in A \cap C_{n+1} \iff f_n \circ f(x) \in \neg A \cap C_n),\]
and that $f_n \colon  \RR \to C_n$ is such that
\[ \forall x \in C_{n+1} (x \in A \cap C_{n+1} \iff f_n (x) \in A \cap C_n).\]
Having these sequences, we will be able to construct a flip-set 
(Wadge-reducible to $A$) using essentially the same argument contained in the proof 
of Theorem 16 in \cite{andrettamartin}. For every  $z\in \Can$ put $g^z_n =f_n \circ f$ if $z(n) = 1$,
and $g^z_n = f_n$ otherwise.
For every $n \in \omega$ choose some $y_{n+1} \in
C_{n+1}$, and for every $m \leq n$ put
$x^z_{n,m} = g^z_m\circ \dotsc \circ
g^z_n(y_{n+1}) \in C_m$.
If we fix $m$ we get that $g^z_m(x^z_{n,m+1}) = x^z_{n,m}$ for every $n>m$,
and that $\{x^z_{n,m} \mid n \geq m\} \subseteq C_m$ is
a Cauchy sequence with respect to $d_m$ by \eqref{eqCauchy}. 
Therefore we can put $x^z_m  = \lim_{n \to \infty} x^z_{n,m}$
and notice that $x^z_m \in C_m$ by the fact that $C_m$ is $\tau_m$-closed,
and that $g^z_m(x^z_{m+1}) =
x^z_m$ by continuity of $g^z_m$. 
Now it is easy to verify
that $F = \{ z \in \Can \mid x^z_0 \in
A\}$
is a flip-set.\\

The construction of the required sequences  will be carried out
by induction on $n$. To reach this goal we will construct also two
auxiliary sequences 
\[ \seq{\mathcal{P}_n}{n \in \omega},\ \seq{\mu_n}{n \in \omega}\]
such that:
\begin{enumerate}[1)]
\item $\mu_n$ is an increasing  sequence of ordinals smaller than $\xi$ and
$f_n$ is a $\bDelta^0_{\mu_n}$-function;
\item $\tau_n$ admits a countable basis $\mathcal{B}_n \subseteq
  \bDelta^0_{\mu_n}$; 
\item $\mathcal{P}_n = \seq{D^n_m}{m \in \omega}$ is a 
$\bDelta^0_{\mu_n}$-partition 
of $\RR$ (in particular is bounded in $\Delta_\F$), 
$\mathcal{P}_{n+1}$ refines 
$\mathcal{P}_n$, $C_n = D^n_m$ for some $m$, and each $D^n_m$ is 
$\tau_n$-clopen.
\end{enumerate}
At stage $n$ we will define $C_n$, $\mathcal{P}_n$,
$f_n$ together with $d_{n+1}$ and $\mu_{n+1}$. 
First let $C_0 = \RR$, $\mathcal{P}_0$ be defined by $D^0_0 = \RR$ and
$D^0_{m+1} = \emptyset$,  
$f_0 = \id$,
$\mu_0 = 1$, and $d_0$ be the usual metric on 
$\RR$. By \textbf{$\sigma$-BC} there is
some $\mu_1 < \xi$ such that $\{f^{-1}(\bN_s) \mid s \in \seqo\} \subseteq
\bDelta^0_{\mu_1}$, and we can let $d_1$ be the metric obtained applying 
Lemma \ref{lemmarel} to this collection of sets 
(so that $f = f_0 \circ f \colon  (\RR,\tau_1)
\to (\RR,\tau_0)$ is continuous). For the inductive 
step we need the following claim, which is analogous to Claim 5.3.1 of \cite{mottorosborelamenability}.

\begin{claim}\label{claimclaim}
Let $D \subseteq \RR$ be in $\bDelta^0_\mu$ (for some $\mu
<\xi$). If $A \leq_\F A \cap D$ then there is $g \in \sf{D}_\mu$ such that 
$g \colon  \RR \to D$ and $g$ reduces $A$ to $A \cap D$.
\end{claim}

\begin{proof}[Proof of the Claim]
We can assume $D \neq \emptyset, \RR$ 
and $\neg A \cap D
\neq \emptyset$, as if $D = \RR$ then we can simply take $g$ to be the identity, while if $D = \emptyset$ or $D \subseteq A$ then $A \leq_\F A \cap D = D$ would contradict $A \notin \Delta_\F$.
By the observation above we
have that $A \nleq_\F A \cap D \iff A \cap D <_\F A$, and by Lemma
\ref{lemmabasic} and $\L$-minimality of $A$ in its $\F$-degree we
have that $A \cap D <_\F A \iff A \cap D <_\L
A$. Thus $A \leq_\F A \cap D$  implies that either $A \leq_\L A \cap D$ or, by $\SLOL$, 
$\neg A \leq_\L A \cap D$. If the second
alternative holds, then since $A \cap D \leq_{\sf{D}_\mu} A$ (see Lemma 4.4 in
\cite{mottorosborelamenability}) one
also has $A \leq_{\sf{D}_\mu} \neg A$: thus in every case $A
\leq_{\sf{D}_\mu} A \cap D$. Let $g' \in \sf{D}_\mu$ be a witness
of this fact. Let $k \in \sf{D}_\mu$ be defined by $k(x) = x$ if $x \in D$ and $k(x) = y$ otherwise,
where $y$ is any fixed point in $\neg A \cap D$. Letting $g = k \circ
g'$  it is easy to
check that our claim holds.
\renewcommand{\qedsymbol}{$\square$ \textit{Claim}}
\end{proof}

Now suppose to have constructed all the sequences until stage $n$, that is 
$C_i$, $\mathcal{P}_i$, $f_i$, $d_j$ and $\mu_j$ for $i \leq n$ and $j \leq
n+1$. 
Recall also from Claim 5.3.2 of \cite{mottorosborelamenability} that
for every $m \leq n$ there is a $\bDelta^0_{\mu_m}$-partition $\{
C^i_m \mid i \in \omega\}$ of $\RR$ such that $d_m\text{-}{\rm diam}(C^i_m) <
2^{-n}$ and $C^i_m$ is  $\tau_m$-clopen for every $i \in \omega$.
Fix $s\in {}^{n+1}2$ and let $g^s_k$ be
defined (for every $k \leq n$) by
\[
g^s_k = 
\begin{cases}
f_k \circ f & \text{if }s(n)=1\\
f_k & \text{if }s(n)=0.
\end{cases}
\] 
 Let 
$\seq{D^0_{i,s}}{i \in \omega}$ be an enumeration of
\[ \{ (g^s_0 \circ \dotsc \circ
g^s_n)^{-1}(C^i_0) \cap D^n_m \mid i,m \in \omega \}\]
and for $k<n$ let $\seq{D^{k+1}_{i,s}}{i \in \omega}$ be an enumeration
of
\[ \{ (g^s_{k+1} \circ \dotsc\circ
g^s_n)^{-1}(C^i_{k+1}) \cap D^k_{j,s} \mid i,j
\in\omega \}.\]
Arguing by induction on $k \leq n$, 
it is not hard to see that $\langle D^n_{i,s} \mid 
i \in \omega\rangle$ is a $\bDelta^0_{\mu_{n+1}}$-partition
which refines $\mathcal{P}_n$, and that each
$D^n_{i,s}$ is $\tau_{n+1}$-clopen since,  by
induction on $k<n$ again, one can prove that
$g^s_k \circ \dotsc \circ g^s_n \colon  (\RR,
d_{n+1}) \to (\RR, d_k)$
is continuous (and the $\tau_{n+1}$-clopen sets are contained 
in $\bDelta^0_{\mu_{n+1}}$, as $\B_{n+1} \subseteq 
\bDelta^0_{\mu_{n+1}}$ by inductive
hypothesis).

Now fix an enumeration $\seq{s_l}{l<2^{n+1}}$ of ${}^{n+1} 2$
and inductively repeat
the argument above but using $\seq{D^n_{i,s_l}}{i \in \omega}$ instead of
$\mathcal{P}_n$ at stage $l+1$. Let $\mathcal{P}_{n+1} =
\seq{D^{n+1}_m}{m \in \omega}$ be the
final partition of $\RR$ obtained at stage $2^{n+1}$, and observe that
one has again that $D^{n+1}_m \in \bDelta^0_{\mu_{n+1}}$ and that 
$D^{n+1}_m$ is $\tau_{n+1}$-clopen for every $m
\in \omega$.  Choose $\bar{m} \in \omega$ such that $A \leq_\F A 
\cap D^{n+1}_{\bar{m}}$
(such an $\bar{m}$ must exist by our assumption, since
$\mathcal{P}_{n+1}$ is a  
$\Delta_\F$-bounded partition of $\RR$), put $C_{n+1} = D^{n+1}_{\bar{m}}$, 
and let 
$f_{n+1}$ 
be the function obtained applying  Claim
\ref{claimclaim} to $C_{n+1}$. 
We claim that there is some $\mu_{n+2} \geq\mu_{n+1}$ 
smaller than $\xi$ such
that both
\[ \{ f_{n+1}^{-1}(B) \mid B \in \B_{n+1} \} \subseteq
\bDelta^0_{\mu_{n+2}}\quad 
\text{and} \quad
\{ (f_{n+1} \circ f)^{-1}(B) \mid B \in \B_{n+1} \} 
\subseteq \bDelta^0_{\mu_{n+2}}.\]
In fact, the first part is obvious (since $f_{n+1}$ is a 
$\bDelta^0_{\mu_{n+1}}$-function
and $\B_{n+1} \subseteq \bDelta^0_{\mu_{n+1}}$). For the second part, since 
$\{ f_{n+1}^{-1}(B) \mid B \in \B_{n+1} \} \subseteq \bDelta^0_{\mu_{n+1}}$
 is countable 
and bounded in $\Delta_\F$, by the \textbf{$\sigma$-BC} there must 
be some $\nu < \xi$ such that
\[ \{ f^{-1}(f_{n+1}^{-1}(B)) \mid B \in \B_{n+1} \} \subseteq \bDelta^0_\nu.\]
Put $\mu_{n+2} = \max \{ \mu_{n+1},\nu\}$: it is easy to check that 
$\mu_{n+2}$ is as required.

Finally, apply Lemma \ref{lemmarel} to the collection
\[ \{ f_{n+1}^{-1}(B) , (f_{n+1} \circ f)^{-1}(B) \mid B \in \B_{n+1}\} 
 \subseteq \bDelta^0_{\mu_{n+2}}\]
to get $d_{n+2}$ with the desired properties (in particular, we have that both 
$f_{n+1} \circ f \colon  (\RR,d_{n+1}) \to (\RR,d_n)$
and $f_{n+1} \colon  (\RR,d_{n+1}) \to (\RR,d_n)$ 
are continuous).
It is not hard to check that the sequences inductively constructed in
this way satisfy all the
conditions required, and this conclude the proof for the case when $\F$
is of type III.

Now let us consider the other possibilities for the set of reductions
$\F$: if $\F$ is of type I we can use the same argument as above but
avoiding to construct the $\mu_n$'s,
letting $d_n$ be
the usual metric on $\RR$ for every $n \in \omega$ 
(thus dropping essentially condition iii), 
and constructing the partitions $\mathcal{P}_n$ in such a 
way that for each $n \in \omega$ 
there is some $k_n$
such that each element of $\mathcal{P}_n$, and in particular $C_n$, 
is $\L$-reducible to
$\bN_{0^{(k_n)}}$ (the collection $\Delta_{k_n} = 
\{A \subseteq \RR \mid A \leq_\L
\bN_{0^{(k_n)}}\}$ can be easily seen
to be closed under finite intersections and unions)\footnote{One has also 
to modify the statement of Claim \ref{claimclaim} in the following way: ``Let
$D \subseteq \RR$ be in $\Delta_{k_n}$ (for some $k_n \in \omega$). If
$A \leq_\F A \cap D$ then there is a $g\colon  \RR \to D$ such that $g$ is a $\Delta_{k_n}$-function which reduces $A$
to $A \cap D$''.}.
Finally, if $\F$ is of type II one can repeat the argument above (in a
slightly simpler way) taking advantage of the fact that \emph{every}
countable family of $\bDelta^0_\xi$
sets is bounded in $\bDelta^0_\xi$.
\end{proof}

\begin{remark}
We can completely remove the hypothesis that $A$ is $\L$-minimal in its $\F$-degree and reprove Proposition \ref{propSDPforminimal} assuming only $\ZF+\ACOR + \NFS$ (thus giving essentially a direct proof of Theorem \ref{theorSDP} under a weaker axiomatization) whenever $\F$ satisfies the following property (which is a consequence of \textbf{PC}): if $D \in \Delta_\F$  and $f$ is a constant function then ${\restr{\id}{D}} \cup {\restr{f}{\neg D}}$ is in $\F$. This is because in this case we can compose any reduction $f \in \F$ of $A$ into $A \cap D$ with the function $k$ defined in the proof of Claim \ref{claimclaim} to get that if $A \leq_\F A \cap D$ then there is some $g \in \F$ such that $g \colon \RR \to D$ and $g$ reduces $A$ to $A \cap D$. This fact can then be used to construct the $f_n$'s and conclude the argument exactly in the same way. About the construction, one should just be careful in the inductive step, and check that an  ordinal $\mu_{n+2}$ with the desired properties exists because  both $f_{n+1}$ and $f$ are $\sigma$-bounded in $\Delta_\F$, and the composition of $\sigma$-bounded functions is still $\sigma$-bounded.
\end{remark}

Now we are ready to prove the Strong Decomposition Theorem.

\begin{theorem}\label{theorSDP}
  Let $\F$ be a Borel set of reductions which satisfies the {\rm
    \textbf{$\sigma$-BC}}. Then  $\F$ has the
  {\rm \textbf{SDP}}.
\end{theorem}

\begin{proof}
Assume first that $\F$ is of type I, II or III. Let $A \leq_\F \neg A \notin \Delta_\F$ and
let $B$ be $\L$-minimal in $[A]_\F$: then  $B$ has the Strong Decomposition
Property with
respect to $\F$ by Proposition \ref{propSDPforminimal}, which by Remark
\ref{remSDP} implies that $A$ has the
Strong Decomposition Property with respect to $\F$ as well.

Now assume that $\F$ is not of  type I--III, i.e.\ that $\bDelta^0_{<\xi} \subsetneq \Delta_\F \subsetneq \bDelta^0_\xi$ for some countable $\xi$ (notice that in this case we will not use the \textbf{$\sigma$-BC}).
By Proposition 3.3 of \cite{mottorosborelamenability} we have that $\F
\subseteq \sf{D}_\xi$, thus if $A \leq_\F \neg A$ we have also $A
\leq_{\sf{D}_\xi} \neg A$. By the \textbf{SDP} for $\sf{D}_\xi$,
there must be a $\bDelta^0_\xi$-partition $\seq{D'_n}{n \in \omega}$ of
$\RR$ such that $A \cap D'_n <_{\sf{D}_\xi} A$ for every $n$.
This partition can be refined to a $\bigcup_{\mu < \xi}
\bPi^0_\mu$-partition $\seq{D_n}{n \in \omega}$ of $\RR$ with the same
property, that is such that $A \cap D_n <_{\sf{D}_\xi} A$ for
every $n$. But $\bigcup_{\mu<\xi} \bPi^0_\mu$ is easily seen to
be bounded in $\Delta_\F$, and $A \cap
D_n <_\F A$ by $\SLO^\F$.
\end{proof}

The Strong Decomposition Theorem (together with part \textit{c)} of
Lemma \ref{lemmapartition}) implies that $A \leq_\F \neg A$ if and only
if $A$ has the Strong Decomposition Property  with respect to $\F$,
thus if $\F$ is 
good we can adjoin the condition $A \leq_\F \neg A$ to the equivalents
of Proposition \ref{propequivSDP}.
Moreover, as a corollary of Theorem \ref{theorSDP} one gets also that if $\F$ is
a good Borel reducibility then at limit levels of
uncountable cofinality there is a nonselfdual pair.

\begin{corollary}\label{coruncountable}
  Let $\F$ be a good Borel set of reductions and let $[A]_\F$ be
  a selfdual limit degree. Then $[A]_\F$ is of countable cofinality.
\end{corollary}

\begin{proof}
  Let $\seq{D_n}{n \in \omega}$ be a $\Delta_\F$-bounded partition of $\RR$
  such that $A\cap D_n <_\F A$ and 
\begin{equation}\label{eqdagger}
\tag{$\dagger$} \forall n \in \omega \exists m\in \omega (A \cap D_n <_\F A \cap D_m) 
\end{equation}
(such a
  partition must exist by Theorem 
  \ref{theorSDP} and by the fact that $[A]_\F$ is limit): then $\mathcal{A} = \{ [A \cap D_n]_\F \mid n \in \omega\}$
  witnesses that $[A]_\F$ is of countable cofinality (use condition $\eqref{eqstar}$ and the fact that if $A \cap D_n \leq_\F B$ for every $n$ then $A \cap D_n <_\L B$ by \eqref{eqdagger} and Lemma \ref{lemmabasic}).
\end{proof}

The Strong Decomposition Theorem implies also  that we can compare good Borel
reducibilities with respect to the degree-structures induced by
them.

\begin{theorem}\label{theorequivalence}
  Let $\F$ and $\G$ be two Borel sets of reductions such that $\G$ has
  the {\rm \textbf{SDP}}, $\F$ satisfies the {\rm \textbf{PC}}, and $\Delta_\G
  \subseteq \Delta_\F$. Then for every $A,B \subseteq \RR$
\[ A \leq_\G B \imp A \leq_\F B.\]

In particular, if $\F$ and $\G$ are good Borel reducibilities
then $\F \simeq \G$ if and only if $\Delta_\F = \Delta_\G$.
\end{theorem}

The proof is identical to the one of Theorem 4.7 in
\cite{mottorosborelamenability} --- the only obvious modification is 
that we have 
to use \textbf{SDP} instead of \textbf{DP}.
Using Theorem \ref{theorequivalence}, one can now obtain the dichotomy theorem for good Borel reducibilities (which is 
simply a more detailed
recasting of Theorem \ref{theordic1}).

\begin{theorem}\label{theordic}
  Let $\F$ be a good Borel reducibility. Then one of
  the following holds:
  \begin{enumerate}[i)]
  \item $\F$ induces a Wadge-like degree-structure:
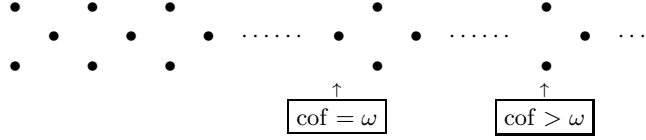
\begin{small}
\begin{equation*}
\begin{array}{llllllllllllll}
\bullet & & \bullet & & \bullet & & & & \bullet & & & \bullet
\\
& \bullet & & \bullet & & \bullet & \cdots \cdots & \bullet
& & \bullet &\cdots\cdots & & \bullet & \cdots
\\
\bullet & & \bullet & & \bullet & & & & \bullet & & & \bullet
\\
& & & & & & &
\, \makebox[0pt]{$\stackrel{\uparrow}{\framebox{${\rm cof} = \omega$}}$}
& & & &
\, \makebox[0pt]{$\stackrel{\uparrow}{\framebox{${\rm cof} > \omega$}}$}
\end{array}
\end{equation*}
\end{small}
  \item $\F$ induces a Lipschitz-like degree-structure:
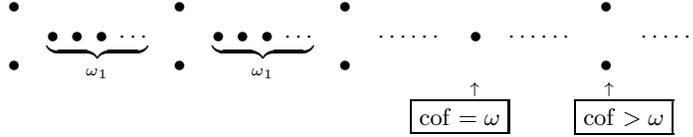
\begin{small}
\begin{equation*}
\begin{array}{llllllllll}
\bullet & & \bullet & & \bullet & & & &\bullet
\\
& \smash[b]{\underbrace{\bullet \; \bullet \;
\bullet \; \cdots}_{\omega_1}} & &
\smash[b]{\underbrace{\bullet \; \bullet \; \bullet \;
\cdots}_{\omega_1}} & & \cdots\cdots &
\bullet & \cdots\cdots & & \cdots\cdots
\\
\bullet & & \bullet & & \bullet & & & & \bullet
\\
& & & & & &
\stackrel{\uparrow}{\makebox[0pt][r]{\framebox{${\rm cof} = \omega$}}} & &
\, \stackrel{\uparrow}{\makebox[0pt][l]{\framebox{${\rm cof} > \omega$}}}
\end{array}
\end{equation*}
\end{small}
  \end{enumerate}
In particular, the first alternative holds if $\F$ is of type {\rm II},
while the second alternative holds if $\F$ is either of type {\rm I} or of
type {\rm III}.
\end{theorem}

The proof of this theorem can be obtained by choosing some ``canonical''
representative for each equivalence class induced by the equivalence
relation $\simeq$ on $\sf{GR}$, and by studying
the degree-structure induced by it. These examples are, respectively:
$\Lip$  
for the collection of the good Borel reducibilities of
type I, $\sf{D}_\xi$ or $\Bor$ for the $\F$'s of type II such that $\Delta_\F =
\bDelta^0_\xi$ (for each $\xi \leq \omega_1$), and the chain of
reductions $\bigcup_{\mu < \xi} \sf{D}_\mu$ for the $\F$'s of type III
such that $\Delta_\F = \bDelta^0_{<\xi}$ (for every countable limit
$\xi$). The degree-structures of $\Bor$ and $\sf{D}_\xi$ have already been 
determined in \cite{mottorosborelamenability}, while the
degree-structures of $\Lip$ and of the chains of reductions will be
determined in the next section of this paper (one can
check that all these results are 
coherent with the description given in Theorem \ref{theordic}). Therefore it
will be  enough to apply Theorem
\ref{theorequivalence}, with $\G$ being the suitable ``canonical''
representative (i.e.\ the canonical example such that $\Delta_\F=
\Delta_\G$), to get the result for an arbitrary good Borel 
reducibility $\F$.

\section{Good Borel reducibilities of 
type I and III} 
\label{sectionLipandchains}

In this section we will analyze the degree-structures induced by $\Lip$ and by (regular) chains of reductions, showing in particular that they are all Lipschitz-like. This will complete the proof of Theorem \ref{theordic}.

\subsection{Lipschitz functions}

First we want to prove that $\Lip$ is a \emph{good} Borel reducibility of \emph{type} I, and this  practically amounts to compute that 
\[ \Delta_\Lip = \bigcup_{0 \neq n \in \omega} [\bN_{0^{(n)}}]_\L \cup
\{\emptyset,\RR\} = \bigcup_{s \in \seqo} [\bN_s]_\L \cup \{ \emptyset \}. \]
One direction is obvious, so we will just prove $\Delta_\Lip \subseteq \bigcup_{0 \neq n \in \omega} [\bN_{0^{(n)}}]_\L \cup
\{\emptyset,\RR\}$. Let $\emptyset \neq A \in \Delta_\Lip$: by definition there are $f \in \Lip$ and $n \in \omega$ such that $f \in \Lip(2^n)$ and $f^{-1}(\bN_{\langle 0 \rangle }) = A$. 
We want to show that $S  = \{ s \in {}^{n+1} \omega \mid f(\bN_s) \subseteq \bN_{\langle 0 \rangle } \}$ is 
such that $A = \bigcup_{s \in S} \bN_s$: since the set on the right of the equation is clearly $\L$-reducible to $\bN_{0^{(n+1)}}$, this will finish the proof. Clearly $\bigcup_{s \in S} \bN_s \subseteq A$. For the other direction, pick any $x \in A$: being $f$ a reduction of $A$ into $\bN_{\langle 0 \rangle }$, $f(x) \in \bN_{\langle 0 \rangle }$. Since $f \in \Lip(2^n)$, $d(f(x),f(y)) \leq 2^{-1}$ for every $y \in \bN_{\restr{x}{(n+1)}}$, which means $f(\bN_{\restr{x}{(n+1)}}) \subseteq \bN_{\langle 0 \rangle }$: but then $\restr{x}{(n+1)} \in S$ and hence $x \in \bigcup_{s \in S} \bN_s$.

Since we have just proved that $\Lip \in \sf{GR}$, to determine the degree-structure induced by $\Lip$ we have only to understand what happens after a selfdual degree. Given any set $A \subseteq
\RR$  define
\[ s_\Lip(A) = \bigoplus\nolimits_n 0^{(n)} \conc A.\]
We want to prove that if $A \leq_\Lip \neg A$ then $[s_\Lip(A)]_\Lip$
is selfdual and is the immediate successor of
$[A]_\Lip$. This  will prove
that after a selfdual $\Lip$-degree there is always another selfdual
$\Lip$-degree, and that $\Lip$ induce a degree-structure which is Lipschitz-like.

\begin{proposition}\label{propLipsuccessor}
Let $A \subseteq \RR$ be $\Lip$-selfdual. Then $s_\Lip(A) \leq_\Lip \neg s_\Lip(A)$,
$A <_\Lip s_\Lip(A)$ and there is no $B$ such that $A <_\Lip B <_\Lip s_\Lip(A)$.
\end{proposition}

\begin{proof}
Let $A$ be $\L$-minimal in its $\Lip$-degree and observe that one has
$A \leq_\L\neg A$ by Proposition \ref{propequivSDP} (note that obviously  $\Lip(2) \subseteq \Lip$). This implies that $A <_\L 0 \conc A
<_\L \dotsc <_\L 0^{(n)} \conc A <_\L \dotsc$, and hence that $s_\Lip(A) \leq_\L \neg
s_\Lip(A)$. Moreover it is clear that for every $n \in \omega$ we have $A \leq_\L
0^{(n)} \conc A$ and that $0^{(n)} \conc A \leq_\Lip A$ via a function
$f \in \Lip(2^n)$. If $B <_\Lip s_\Lip(A)$ we have that $B <_\L s_\Lip(A)$ by
Lemma \ref{lemmabasic}, which in turn implies $B \leq_\L 0^{(n)} \conc A$ for
some $n \in \omega$: hence $B \leq_\Lip A$. Therefore it remains
only to prove that $s_\Lip(A) \nleq_\Lip A$. Toward a contradiction, assume that
there is $f \in \Lip$ such that $f^{-1}(A) = s_\Lip(A)$, and let $n$ be the
smallest natural number
such that $f \in \Lip (2^n)$, so that $f(\bN_{0^{(n+1)}}) \subseteq
\bN_{\langle k \rangle}$ for some $k \in \omega$. Let $g$ be defined by $g(x)
= f(x)$ if $x \in \bN_{0^{(n+1)}}$ and $g(x) = (k+1) \conc \vec{0}$ otherwise:
then $g \in \Lip(2^{n+1})$ and reduces $0^{(n+1)} \conc A$
to $A \cap \bN_{\langle k \rangle}$. But it is easy to check that $A \cap
\bN_{\langle k \rangle} \leq_\Lip \loc{A}{k}$ and 
$\loc{A}{k} <_\L A$: therefore, by $\L$-minimality of $A$ in its $\Lip$-degree
we would have that
\[ 0^{(n+1)} \conc A \leq_\Lip \loc{A}{k} <_\Lip A,\]
a contradiction!
\end{proof}

The definition of the successor
 operator $s_\Lip$,
allows also to obtain a way to construct the $\Lip$-degrees from the
$\L$-degrees. In fact, if
$[A]_\Lip$ is nonselfdual, then $[A]_\Lip = [A]_\L$ by Theorem
\ref{theorgeneralproperties}, while if $A \leq_\L \neg A$
and $A$ is $\L$-minimal in its $\Lip$-degree, then $[A]_\Lip$ is exactly
$\bigcup_{n \in \omega} [0^{(n)} \conc A]_\L$.\\

As an application of Theorem \ref{theorequivalence}, let us now consider the set of the uniformly continuous functions
(which will be denoted by $\sf{UCont}$):  it turns out (perhaps rather
surprisingly, since uniform
continuity is just a weak ``refinement'' of continuity) that
$\sf{UCont}$ is equivalent to $\Lip$ (rather than to $\W$),
and thus gives a hierarchy of degrees which is Lipschitz-like.
In fact, one can easily check that $\sf{UCont}$ is a good Borel
reducibility and that $\Delta_{\sf{UCont}} = \bigcup_{s \in
{}^{<\omega}\omega}[\bN_s]_\L \cup \{ \emptyset \}$: $\bN_s$ is reducible
to $\bN_{\langle 0
\rangle}$ via a function in $\Lip(2^{\leng(s)}) \subseteq
\sf{UCont}$,  while if
$f$ is uniformly continuous then there must be some $m \in \omega$ such that
for every $x,y \in \RR$
\[ d(x,y) \leq 2^{-m} \imp d(f(x),f(y)) \leq 2^{-1},\]
and thus, in particular, $f$ can not reduce $\bigoplus_n
\bN_{0^{(n)} }$ to $\bN_{\langle 0 \rangle}$ (the argument is similar
to the one used in Proposition \ref{propLipsuccessor}). This proves also that
$\Delta_\Lip = \Delta_{\sf{UCont}}$, and that $\sf{UCont}$ is of type
I: therefore $\Lip \simeq \sf{UCont}$ by Theorem
\ref{theorequivalence}. Moreover it is not hard to check that
$\sf{UCont}$ is maximal among the good Borel reducibilities
of type I, since the fact that $\F$ is of type I and satisfies the
\textbf{$\sigma$-BC} implies $\F \subseteq \sf{UCont}$ --- $\sf{UCont}$ is
exactly the collection of all $\sigma$-bounded $\Delta_\Lip$-functions.

\subsection{Chains of reductions}

A (countable) \emph{chain of (Borel-amenable sets of)
reductions} is simply any sequence $\vec{\F}
= \langle
\F_n\mid n \in \omega \rangle$ of Borel-amenable sets of
reductions. 
To each chain of reductions
we will associate the unique sequence of ordinals $\langle\mu_n\mid n
\in \omega \rangle$
such that $1 \leq \mu_n \leq \omega_1$ and $\Delta_{\F_n} =
\bDelta^0_{\mu_n}$  for every $n \in \omega$, which will
be called the \emph{type of $\vec{\F}$}.
Moreover we will say that
$\vec{\F}$ is of \emph{rank}
$\omega_1$ if $\mu_n = \omega_1$ for some $n \in
\omega$, and of \emph{rank} $1 \leq \xi < \omega_1$ if $\mu_n < \omega_1$
for every
$n \in \omega$ and $\xi =
\sup\{\mu_n+1\mid n \in \omega \}$.
A chain of reductions will be called \emph{regular} if each $\F_n$ is saturated and $\F_n \subsetneq \F_{n+1}$ for every $n$ (in particular, the rank $\xi$ of $\vec{\F}$ must be countable and limit\footnote{It is easy to check that each chain of reductions is equivalent either to a Borel-amenable set of reductions (if it has successor or uncountable rank) or to a regular chain of reductions with the same rank.}). Note that
in this case
$\bigcup_n \F_n = \bigcup_{\mu<\xi} \sf{D}_\mu$ is a Borel set of reductions, and since
\[ \Delta_{\bigcup_n \F_n} = \bigcup\nolimits_n \Delta_{\F_n} = \bigcup\nolimits_n \bDelta^0_{\mu_n} = \bDelta^0_{<\xi} \]
one can check that $\bigcup_n \F_n$ is good and of type III: thus, as we have already pointed out, regular chains of reductions
provide a canonical way to construct good Borel
reducibilities of type III (one for each possible characteristic
set). 
From now onward, we will fix some limit $\xi < \omega_1$ and consider a regular chain of reductions $\vec{\F} = \seq{\F_n}{n \in \omega}$ of rank $\xi$.
By Corollary \ref{coruncountable},
in order to describe the structure
of degrees induced by\footnote{For simplicity of notation, from now on we will systematically identify $\vec{\F}$ with $\bigcup_n \F_n$ when there is no possibility of misunderstanding.} $\leq_{\vec{\F}}$  we have only to determine what happens
after a selfdual degree:
this can be done using
the following proposition about Borel-amenable sets of reductions.

\begin{proposition}\label{propsuccessor}
Let $\G$ and $\G'$ be two Borel-amenable sets of reductions such that
$\Delta_\G
\subsetneq \Delta_{\G'}$ (i.e.\ such that $\G$ is of level strictly smaller
than
$\G'$). Let $A \leq_\G \neg A$ and $B$ be a (nonselfdual) successor
of $A$ with respect to $\leq_\G$: then $B \leq_{\G'} A$.
In particular, if $\mu$ is the level of $\G$ and $A \leq_\G \neg A$, then
$\Sigma^\mu(A) \leq_{\G'} A$ and $\Pi^\mu(A) \leq_{\G'} A$.
\end{proposition}

\begin{proof}
If $\G$ and $B$ are as above
then either $B \equiv_\G \Sigma^\mu(A)$ or $B \equiv_\G \Pi^\mu(A) \equiv_\G \neg \Sigma^\mu(A)$,
and since $\Delta_\G \subseteq \Delta_{\G'}$ implies $A \leq_\G B \imp
A \leq_{\G'} B$ for every $A,B \subseteq \RR$ (by Theorem \ref{theorequivalence}), it is enough to prove
$\Sigma^\mu(A) \leq_{\G'} A$. Let $P_n$
and $R_\mu$ be the sets used to define the operation $\Sigma^\mu$,
and define $F_0 = \{ x\in \RR\mid  \pi_0(x) \in P_0\}$ and
$F_{n+1} = \{x\in \RR\mid \pi_{2(n+1)}(x) \in P_{n+1} \wedge \forall
i \leq n (\pi_{2i}(x) \notin P_i)\}$. Clearly, every $F_n \in
\bDelta^0_\mu \subseteq \Delta_{\G'}$, and since $R_\mu \in \bPi^0_\mu$
and $\bDelta^0_\mu \subsetneq \Delta_{\G'}$, we have also $R_\mu \in
\Delta_{\G'}$ (as $\bSigma^0_\mu \cup \bPi^0_\mu \subseteq \Delta_{\G'}$ by Borel-determinacy). On each of these sets we can continuously reduce
$\Sigma^\mu(A)$ to $A$ using $\pi_{2n+1}$ on the $F_n$'s and a
constant function with value $\bar{y} \notin A$ on $R_\mu$ ($A \neq \RR$ as $A \leq_\G \neg A$), hence $\Sigma^\mu(A)
\leq_{\sf{D}^\W_{\mu'}} A$, where $\mu'$ is the level of $\G'$. But
since $\G' \simeq \sf{D}^\W_{\mu'}$ we are done.
\end{proof}

\begin{theorem}\label{theorsuccessor}
  If $A \leq_{\vec{\F}} \neg A$ then there is some $B \leq_{\vec{\F}} \neg B$
  with the property that  $A <_{\vec{\F}} B$ and there is no $C$ such that $A <_{\vec{F}} C <_{\vec{\F}} B$.
Thus after an $\vec{\F}$-selfdual degree there is another $\vec{\F}$-selfdual degree.
\end{theorem}

\begin{proof}
Taking $A$ to be $\L$-minimal in $[A]_{\vec{\F}}$, by Proposition
\ref{propequivSDP} and the fact that $\vec{\F}$ has the \textbf{SDP} we can
assume $A \leq_\L\neg A$ (hence, in
particular, $A \leq_{\F_n} \neg A$ for every $n \in \omega$).
Let $\seq{\mu_n}{n \in \omega}$ be the type of $\vec{\F}$ and
define the \emph{successor operator} $s_{\vec{\F}}$ by letting
\[ B = s_{\vec{\F}}(A) = \bigoplus\nolimits_n \Sigma^{\mu_n}(A).\]
Clearly $A \leq_\L s_{\vec{\F}}(A)$, and if $C<_{\vec{\F}} s_{\vec{\F}}(A)$ then we have also $C <_\L s_{\vec{\F}}(A)$,
which implies $C \leq_\L \Sigma^{\mu_n}(A)$ for some $n \in \omega$: but
since $\Sigma^{\mu_n}(A) \leq_{\F_{n+1}} A$ by Proposition
\ref{propsuccessor}, we have also $C \leq_{\vec{\F}} A$. Finally, the fact
 that $s_{\vec{\F}}(A) \leq_{\vec{\F}} \neg s_{\vec{\F}}(A)$ will follow from  the fact
that $\Sigma^{\mu_n}(A) <_\L \Sigma^{\mu_{n+1}}(A)$ for every $n \in
\omega$  (since  this implies $s_{\vec{\F}}(A)\leq_\L \neg s_{\vec{\F}}(A)$). To see this, recall
that $\Sigma^{\mu_n}(A) \leq_{\F_{n+1}} A$ while $A <_{\F_{n+1}}
\Sigma^{\mu_{n+1}}(A)$, which implies $\Sigma^{\mu_n}(A)
<_{\F_{n+1}} \Sigma^{\mu_{n+1}}(A)$: hence $\Sigma^{\mu_n}(A) <_\L
\Sigma^{\mu_{n+1}}(A)$ by Lemma \ref{lemmabasic}.
\end{proof}

In particular, Theorem \ref{theorgeneralproperties}, Corollary
\ref{coruncountable} and Theorem \ref{theorsuccessor} implies
that the degree-structure induced by any regular chain of reductions
$\vec{\F}$ (i.e.\ by the preorder $\leq_{\vec{\F}}$) is Lipschitz-like.

\section{Baire reductions}\label{sectionBaire}

Let $\B_\alpha$ (for $\alpha<\omega_1$) denote the set of all Baire
class $\alpha$ functions from $\RR$ into itself,
i.e.\ the set of all functions $f\colon \RR \to \RR$ such that
$f^{-1}(U) \in 
\bSigma^0_{\alpha+1}$ for every open set $U$. Clearly $\sf{D}_1 = \B_0
\subseteq \B_\alpha \subseteq \Bor$ for every $\alpha < \omega_1$, and
$\B_\mu \subseteq \B_\nu$ if and
only if $\mu \leq \nu$. Moreover it is well known that the Baire class $\alpha$
functions provides a stratification of $\Bor$ in $\omega_1$-many levels
which is alternative  to the
one induced by $\bDelta^0_\xi$-functions, thus it is quite natural to try
to study the reducibilities induced by the Baire class functions
(of some level). Unfortunately, if $\alpha \neq 0$ then $\B_\alpha$
is not a set of reductions as it is not closed under composition:
in fact, it is easy to check that
if $f \in \B_\mu$ and $g \in \B_\nu$ then $g \circ f \in
\B_{\mu+\nu}$ and, moreover, there are such an $f$ and $g$ for which $g \circ f
\notin \B_\eta$ for any $\eta < \mu+\nu$. Nevertheless, we can
exactly compute the closure under composition of $\B_\alpha$ 
by reversing the previous composition
law, i.e.\ by showing that if $h \in \B_{\mu+\nu}$ then there are $f \in
\B_\mu$ and $g \in \B_\nu$ such that $h = g \circ f$.
To obtain this computation we will use Lemma \ref{lemmarel}  together
with the following crucial fact (for simplicity of notation
we will put $\bDelta^0_0 = \bDelta^0_1$).

\begin{lemma}[$\ZF+\ACOR$]
For every nonzero $\mu,\nu < \omega_1$ with $\nu > 1$ and every $\mathcal{C} =
\{ C_n \mid n \in
\omega\} \subseteq \bDelta^0_{\mu+\nu}$ there is  $\mathcal{B}  = \{ B_m
\mid m \in \omega \} \subseteq \bDelta^0_{\mu+1}$ such that $\mathcal{C} \subseteq \bDelta^0_\nu(\tau')$ for every topology
$\tau'$ for which $\mathcal{B} \subseteq \bDelta^0_1(\tau')$.
\end{lemma}

\begin{proof}
Clearly we can assume that $\mathcal{C}$ is closed under complementation (if 
not simply adjoin $\neg C_n$ to $\mathcal{C}$ for every $n$). We will
prove the lemma by induction on $\nu$, and the base of the induction
and the successor case will be proved 
together.  Assume $\nu = \eta+1$ (with $\eta\geq 1$): by definition there
must be a collection $\mathcal{D}' = \{ D'_{n,k} \mid n,k \in \omega\}
\subseteq \bPi^0_{\mu+\eta}$ such that $C_n = \bigcup_{k \in \omega} D'_{n,k}$
for every $n$,
and by definition again there must be some $\mathcal{D} = \{ D_{n,k,i} \mid
n,k,i \in \omega \} \subseteq \bDelta^0_{\mu+\eta}$ such that $D'_{n,k} =
\bigcap_{i \in \omega} D_{n,k,i}$ for every $n,k$.
Put $\mathcal{B} = \mathcal{D}$ if $\eta=1$ or, in the other case, use
the inductive 
hypothesis applied to  $\mathcal{D}$ to find  some
countable $\B \subseteq \bDelta^0_{\mu+1}$ such that for every topology $\tau'$
on $\RR$ if $\B \subseteq \bDelta^0_1(\tau')$ then $\mathcal{D} \subseteq
\bDelta^0_{\eta}(\tau')$. In both cases $\mathcal{D}' \subseteq \bPi^0_\eta(\tau')$ and
$\mathcal{C} \subseteq \bDelta^0_{\eta+1}(\tau')$
(by closure under complementation of
$\mathcal{C}$), hence we are done.

Now let $\nu$ be limit and let
$\langle \nu_i \mid i \in \omega \rangle$
be any increasing sequence of ordinals cofinal in $\nu$ such that
$\nu_i > 1$ for every $i$. Since $\mathcal{C}
\subseteq \bDelta^0_{\mu+\nu}$ there must be some $\mathcal{D} = \{
D_{n,k} \mid
n,k \in \omega \} \subseteq \bDelta^0_{<(\mu+\nu)}$ such that $C_n =
\bigcup_{k \in \omega} D_{n,k}$. Put $\mathcal{D}_i = \{ D_{n,k} \in
\mathcal{D} \mid D_{n,k} \in 
\bDelta^0_{\mu + \nu_i}\}$
for every $i$, so that $\mathcal{D} = \bigcup_{i \in \omega} \mathcal{D}_i$.
Applying the inductive hypothesis to each $\mathcal{D}_i$ and using $\ACOR$,
we can find for every $i$ a collection $\B_i \subseteq
\bDelta^0_{\mu+1}$ such that
if $\tau'$ is any topology on $\RR$ for which $\B_i \subseteq
\bDelta^0_1(\tau')$ then $\mathcal{D}_i \subseteq \bDelta^0_{\nu_i}(\tau')$.
Put now $\B = \bigcup_{i \in \omega} \B_i$. Then $\B \subseteq
\bDelta^0_{\mu+1}$ and
if $\tau'$ is such that $\B \subseteq \bDelta^0_1(\tau')$ then $\mathcal{D}
\subseteq \bDelta^0_{<\nu}(\tau')$ and hence $\mathcal{C}
\subseteq \bDelta^0_\nu(\tau')$.
\end{proof}

Recall now the following classical fact: if $X$ is a zero-dimensional Polish
space then there is a closed set $F \subseteq \RR$ and an homeomorphism $H \colon 
F \to X$.

\begin{proposition}[$\ZF+\ACOR$]\label{propBaire}
Let $h \colon  \RR \to \RR$ be in $\B_{\mu+\nu}$ (for some countable ordinals $\mu$
and $\nu$). Then there are $f \in \B_\mu$ and $g \in \B_\nu$ such that
$h = g \circ f$.
\end{proposition}

\begin{proof}
Let $\tau$ be the usual topology on $\RR$.
If $\mu=0$ or $\nu=0$ the result is trivial (simply take $f=\id$ and
$g=h$ or, respectively, $f=h$ and $g=\id$). Hence we can assume
$\mu,\nu > 0$.
Put $\mathcal{C}
 = \{h^{-1}(\bN_s) \mid s \in \seqo \} \subseteq \bDelta^0_{\mu+ \nu +1}$.
Let $\B \subseteq \bDelta^0_{\mu+1}$ be obtained as in the previous lemma,
that is such that for any topology $\tau'$ if $\B \subseteq \bDelta^0_1(\tau')$
then $\mathcal{C} \subseteq \bDelta^0_{\nu+1}(\tau')$. Apply Lemma
\ref{lemmarel} to $\B$ 
in order to obtain a zero-dimensional Polish topology $\tau'$ such that
$\B \subseteq \bDelta^0_1(\tau')$ and let $F\subseteq \RR$ be a closed set such
that $H \colon  (F,\tau) \to (\RR,\tau')$ is an homeomorphism. Finally
let $r \colon  \RR \onto 
F$ be a retraction. Now put $g = h \circ H \circ r \colon  (\RR,\tau) \to
(\RR,\tau)$ and $f = 
H^{-1} \colon  (\RR,\tau) \to (\RR,\tau)$.
It is easy to check that $h \colon  (\RR,\tau') \to (\RR,\tau)$ is of Baire class $\nu$,
and thus also $g$ is of Baire class $\nu$. Moreover, since $H^{-1} \colon 
(\RR,\tau') \to (F,\tau)$ is continuous and $\bDelta^0_1(\tau') \subseteq
\bDelta^0_{\mu+1}(\tau)$, we have that $f$ is of Baire class $\mu$. Thus we have
only to prove that $g \circ f = h$. Since ${\rm range}(H^{-1}) = F$, we have
that $r(H^{-1}(x)) = {\rm id}(H^{-1}(x)) = H^{-1}(x)$ for every $x \in \RR$.
But then
\[ g \circ f(x) = h(H(r(H^{-1}(x)))) = h(H(H^{-1}(x))) = h(x). \qedhere \] 
\end{proof}

Observe that the same statement
is true if we replace $h$ with a $\bSigma^0_{\mu+\nu}$-measurable
function (with $\mu,\nu>1$) and we require that there are a
$\bSigma^0_\mu$-measurable function  and a
$\bSigma^0_\nu$-measurable function whose composition gives $h$.

\begin{remark}
  The previous proposition can be applied also to other Polish
  spaces $\mathscr{X}$ (clearly we can assume again that $\mu,\nu \neq
  0$, otherwise the result is trivial). In fact the same argument  shows
  that for every 
  $h \colon  \mathscr{X} \to \mathscr{X}$ of Baire class $\mu+\nu$ 
  there are  $f \colon  \mathscr{X} \to \RR$ of Baire class $\mu$ and $g
  \colon  \RR \to \mathscr{X}$ of Baire class $\nu$ such that $h = g \circ
  f$. Moreover, if we assume that $\mathscr{X}$ is (uncountable
  and) not $\mathbf{K}_\sigma$, the same result remains true also
  replacing $f$  and $g$
  with two functions $f' ,g' \colon \mathscr{X}
  \to \mathscr{X}$ of Baire class $\mu$ and $\nu$, respectively. In
  fact in this case there is a closed
  set $F$ of $\mathscr{X}$ which is homeomorphic to $\RR$ via some
  function $H'$, hence one can define $f' = H'^{-1} \circ f$ and $g' = (\restr{g \circ
  H'}{F}) \cup (\restr{f_0}{\mathscr{X} \setminus F})$, where $f$ and
$g$ are  obtained as in the previous proof and $f_0$ is any constant
function, and check that they are still of the correct Baire
class. Finally, this last version of  Proposition \ref{propBaire} can be further
extended to \emph{every} uncountable Polish space $\mathscr{X}$ if we
assume $\nu \neq 1$: in fact in this case we can use the fact that
every zero-dimensional Polish space is homeomorphic to some
$\mathbf{G}_\delta$ subspace of the Cantor space $\Can$, which is in turn
homeomorphic to a closed subset of $\mathscr{X}$. Therefore any
zero-dimensional Polish space is homeomorphic to a $\mathbf{G}_\delta$
set $G$ of $\mathscr{X}$ via some function $H'$, and we can define
$f'$ and $g'$ as above but replacing $F$ with $G$.
\end{remark}

\begin{theorem}[$\ZF+\ACOR$]
Let $\alpha$ be a nonzero countable ordinal. Then the closure under
composition of $\B_\alpha$ is exactly $\bigcup_{\mu < \xi} \B_\mu$, where
$\xi = \alpha \cdot \omega$ is the least additively closed ordinal above
$\alpha$.
\end{theorem}

\begin{proof}
One direction is trivial (since the composition of $n$ Baire class
$\alpha$ functions is in 
$\B_{\alpha \cdot n}$). Suppose now that $h$ belongs to $\B_{\alpha
  \cdot n}$ for some $1 \leq n 
\in \omega$. We will prove by induction on $n$ that $h$ belongs to the
closure under composition of $\B_\alpha$. If $n=1$  
there is nothing to prove, while if $n = m+1$ (for some $m \geq
1$) then $h \in \B_{\alpha \cdot m + \alpha}$ and we can apply
Proposition \ref{propBaire} to get $f \in \B_{\alpha \cdot m}$ and $g \in
\B_\alpha$ such that $h = g \circ f$. Applying now the inductive
hypothesis to $f$ we get the result.
\end{proof}

By the previous theorem, we are naturally led to take any countable additively closed
ordinal $\xi$ (recall that $\xi$ is additively closed if and only if
  either $\xi = 0,1$ or $\xi = \omega^\mu$ for some ordinal $\mu$) and
study the degree structure induced by
\[ \sf{B}_\xi = \B_{< \xi} = \bigcup\nolimits_{\mu < \xi} \B_\mu\]
(by the rule of composition above, $\sf{B}_\xi$ is closed
under composition and hence it is a Borel set of reductions).
Since it is straightforward to check that $\sf{B}_\xi$
is a good Borel reducibility (and therefore has the
\textbf{SDP}), and that
$\Delta_{\sf{B}_\xi} = \bDelta^0_{<\xi}$
(which in particular implies that $\sf{B}_\xi$ is of type III), we can apply Theorem
\ref{theorequivalence} to get that $\sf{B}_\xi$ is equivalent to any (regular)
chain of reductions of rank $\xi$, and thus induces the same degree-structure.

This equivalence is non trivial
(at least for $\xi=\omega$): in fact we will show that $\sf{B}_\omega$ is not contained
in $\bigcup_{n \in \omega}\sf{D}_n$ by proving that there is a Baire
class $1$ function which is not in $\sf{D}_n$ for any $n \in
\omega$ (this discussion will also cover the missing proofs
about the Pawlikowski function in
\cite{mottorosborelamenability}).
First let us recall the definition of the
 Pawlikowski
 function $P$ from \cite{solecki}. Let $\omega+1$ have the order
topology and consider the space ${}^\omega(\omega+1)$
endowed with the corresponding product topology. It is easy to check that ${}^\omega (\omega+1)$ is perfect, zero-dimensional
and compact, hence it is homeomorphic to the Cantor space $\Can$. Let $\gamma\colon \omega+1 \to \omega$
be the
bijection defined by $\gamma(\omega) = 0$ and $\gamma(n) = n+1$ for
any $n \in \omega$, and define $P \colon  {}^\omega(\omega+1) \to \RR$
using $\gamma$ coordinatewise, i.e.\ putting $P(x) = \langle
\gamma(x(n))\mid n \in \omega \rangle$. Define also (again coordinatewise)
 a ``partial'' function
\[\hat{\gamma}\colon  {}^{<\omega}(\omega+1) \to {}^{<\omega}\omega\colon s \mapsto
\langle \gamma(s(i))\mid i < \leng(s) \rangle,\]
and note that both $P$ and $\hat{\gamma}$ are bijection between the corresponding
spaces.

Given
$\tau \in {}^{<\omega}(\omega+1)$, consider the set $C_\tau =\{x \in
{}^\omega(\omega+1)\mid \tau \subseteq x\}$: by simple arguments, it turns out
that $C_\tau$ is always a closed set, has empty interior if and only if there
is some $i<\leng(\tau)$ such that $\tau(i)=\omega$, and is also open (hence
a clopen set) if and only if $\tau(i) \neq \omega$ for every $i<\leng(\tau)$.
In particular, this implies that for every $n \in \omega$ the set $K_n = \{x
\in {}^\omega(\omega+1)\mid x(n)=\omega\}$ is a closed set with empty interior
and hence a nowhere dense proper closed set (from this fact one can also
derive that ${}^\omega\omega$ is a proper $\mathbf{G}_\delta$ subset
of ${}^\omega 
(\omega+1)$ which is also comeager and dense in it).

\begin{lemma}[$\ZF+\ACOR$]\label{lemmapropagation}
  Let $\mathscr{X}$ be any zero-dimensional space. Let $\alpha < \omega_1$ be a
  nonzero ordinal and let $\langle \alpha_n\mid 
n \in \omega \rangle$ be an increasing  sequence of ordinals
smaller than $\alpha$ and
cofinal in it. For every family of sets $\{P_n \subseteq {}^\omega
\mathscr{X} \mid n
\in \omega \}$ such that $P_n$ is $\bPi^0_{\alpha_n}$-complete (for
every $n \in 
\omega$), the set $S \subseteq {}^\omega \mathscr{X}$ defined by
\[S = \{x \in {}^\omega \mathscr{X} \mid \exists n
(\pi^\mathscr{X}_n(x) \in P_n)\}\]
is a $\bSigma^0_\alpha$-complete set.
In particular, if $P \subseteq {}^\omega \mathscr{X}$ is a
$\bPi^0_\alpha$-complete 
set then $S=\{x \in {}^\omega \mathscr{X}\mid \exists
n(\pi^\mathscr{X}_n(x)\in P)\}$ is a
$\bSigma^0_{\alpha+1}$-complete set.
\end{lemma}

\begin{proof}
  Clearly $S \in \bSigma^0_\alpha$. Let now $Q$ be any set in
$\bSigma^0_\alpha$: by definition there is an increasing sequence $\langle
\beta_k\mid k\in\omega \rangle$ of ordinals smaller than $\alpha$ such that
$Q =\bigcup_k R_k$ for some sets $R_k \in \bPi^0_{\beta_k}$. Since the sequence
$\langle \alpha_n\mid n \in \omega \rangle$ is increasing and cofinal in $\alpha$ we can find
a subsequence $\langle \alpha_{n_k}\mid k\in\omega \rangle$ such that
$\beta_k \leq \alpha_{n_k}$ for any $k \in \omega$. Moreover, since every $P_n$
is $\bPi^0_{\alpha_n}$-complete we can choose a sequence of points $\langle
x_n\mid n \in\omega \rangle$ such that $x_n \notin P_n$ for every $n \in
\omega$. Now define a sequence of continuous functions $\langle f_n\mid
n \in \omega \rangle$ by letting $f_{n_k}$ be any continuous
reduction of $R_k$  in
$P_{n_k}$ (which exists since 
$R_k \in \bPi^0_{\beta_k} \subseteq \bPi^0_{\alpha_{n_k}}$), and $f_n$
be the 
constant function with value $x_n$ if there is no $k \in \omega$ such that
$n=n_k$. Finally, put $f=\bigotimes\nolimits^\mathscr{X}_n f_n$. Clearly $f$ is
continuous and it is not hard to check that it reduces $Q$ to $S$, i.e.\ that $x \in Q \iff
f(x) \in S$ for every $x \in {}^\omega \mathscr{X}$.
\end{proof}

Now we are ready to prove the
following proposition which gives the exact complexity of $P$.

\begin{proposition}[$\ZF+\ACOR$] \label{propPcomplexity}
The function $P$ is of Baire class $1$ and is in $\sf{D}_\omega$
but not in $\sf{D}_n$ (for any nonzero $n \in \omega$).
\end{proposition}

\begin{proof}
  Since $P^{-1}(\bN_s) = C_{\hat{\gamma}^{-1}(s)}$ for every $s \in
  {}^{<\omega}
\omega$, we have that $P^{-1}(U)$ is the union of countably many closed sets
for any open set $U \subseteq \RR$: hence $P$ is of Baire class $1$ (this also
implies that $P \in \sf{D}_\omega$ since every Baire class $n$
function is in $\sf{D}_\omega$).
It remains only to prove that $P$ is not
in $\sf{D}_n$ for any $n \in \omega$. First define
\begin{align*}
S_1 &= \{x \in \RR\mid \exists n (x(n)=0)\}\\
S_{n+1}&= \{ x \in \RR \mid \exists n (\pi_n(x) \notin S_n\}
\end{align*}
for $n \geq 1$. One can inductively check
that $S_n\subseteq \RR$ is a $\bSigma^0_n$
set, and that $P^{-1}(S_n) \subseteq {}^\omega(\omega+1)$ is a complete
(and hence also proper) $\bSigma^0_{n+1}$
set (use Lemma
\ref{lemmapropagation} 
for the inductive step). Passing to the complements, we have that
$\neg S_n \in 
\bPi^0_n \subseteq \bDelta^0_{n+1}$ but $P^{-1}(\neg S_n) \notin
\bDelta^0_{n+1}$, i.e.\ $P$ is not a $\bDelta^0_{n+1}$-function.
\end{proof}

The function $P$ is defined from ${}^\omega (\omega+1)$ to $\RR$, while we are
interested in functions from $\RR$ into itself. Nevertheless it is easy to see
how to obtained from $P$ a function $\hat{P} \colon  \RR \to \RR$ with the same
complexity. Let  $h\colon {}^\omega(\omega+1) \to \Can$ be any homeomorphism between
${}^\omega(\omega+1)$ and the
Cantor space $\Can$ (which is a closed subspace of $\RR$). Define
$\hat{P}\colon  \RR \to \RR$ by letting $\hat{P}(x) =
P(h^{-1}(x))$ if  $x \in \Can$ and $\hat{P}(x)=\vec{0}$ otherwise.
Following
\cite{solecki}, for every $f \colon X_1 \to Y_1$
and $g \colon X_2 \to Y_2$ put $f 
\sqsubseteq g$ just in case there are two embeddings $\fhi \colon
X_1 \to X_2$ and $\psi \colon f(X_1) \to Y_2$ such that $\psi
\circ f = g \circ \fhi$.
Clearly, if $f \sqsubseteq g$ and $g$ is a
$\bDelta^0_\xi$-function (respectively, a Baire class $\xi$ function) then
also $f$ is a $\bDelta^0_\xi$-function (respectively a Baire class $\xi$
function), and therefore if $f$ is \emph{not} a
$\bDelta^0_\xi$-function 
(resp.\ a Baire class $\xi$ function) then neither $g$ is a
$\bDelta^0_\xi$-function (resp.\ a Baire class $\xi$ function).
Since it is not hard to prove that $\hat{P}$ is still a Baire class
$1$ function 
and  that $h$ and the
identity
function witness $P \sqsubseteq \hat{P}$, 
we have that $\hat{P} \in \sf{D}_\omega$ but $\hat{P}
\notin \sf{D}_n$ for any $n \in \omega$, hence we are done.\\

\section{Comparing hierarchies of degrees}\label{sectioncomparing}

\emph{All sets of functions considered in this section are assumed to be good
Borel reducibilities}.
Let $\G_\mu$ denote an arbitrary good Borel set of reductions with
$\Delta_{\G_\mu} = \bDelta^0_\mu$ (in particular $\G_\mu$ is always of
type II). 
To clarify the relationship between (the degree-structures induced by)
different good Borel  reducibilities, 
note that each $\F$ of type I induces the finest possible hierarchy
(in particular finer than the hierarchy induced by any $\G_1$, which
is in some sense the next ``level of reducibility''), and each
$\mathcal{H}$ of type III with $\Delta_{\mathcal{H}} =
\bDelta^0_{< \xi}$ (for $\xi$ a countable limit ordinal) induces an
hierarchy of degrees which is coarser than the hierarchy of the
$\G_\mu$-degrees (for any $\mu<\xi$), and finer than the hierarchy of
the $\G_\xi$-degrees. Finally $\Bor$, and the sets of reductions with
the same characteristic set, gives the coarsest hierarchy. By part
vi)  of Theorem \ref{theorgeneralproperties}, it is clear that for an
$\F$-hierarchy being coarser than the $\F'$-hierarchy amount to the
fact that the $\F$-selfdual degrees are obtained gluing together many
$\F'$-degrees: therefore to understand how the $\F$-structure can be
obtained from the finer ones we must describe how each $\F$-selfdual
degree is constructed. 

The first case, that is when we want to compare the $\F$-hierarchy
(for $\F$ of type I) with the $\G_1$-hierarchy, is clearly solved by
the Steel--Van Wesep Theorem, which says that ${A \leq_\W \neg A} \iff
{A \leq_\L \neg A}$: in fact since $\L \subseteq \F \subseteq \W
\simeq \G_1$ it must be the case that $A \leq_{\G_1} \neg A$ if and
only if $A \leq_\F \neg A$, and as $s_\Lip(A) \leq_\W A$ we get that
each $\G_1$-selfdual degree is exactly the union of a (maximal)
$\omega_1$-chain of consecutive $\F$-selfdual degrees. 

Now consider $\mathcal{H}$ of type III as above: as $\mathcal{H}
\simeq \bigcup_{\mu < \xi} \G_\mu$, it is clear that if $A
\leq_{\mathcal{H}} \neg A$ then $[A]_{\mathcal{H}} = \bigcup_{\mu <
  \xi } [A]_{\G_\mu}$. Therefore the $\mathcal{H}$-hierarchy is the
minimal degree-structure which is refined by all the
$\G_\mu$-structures.  

Finally, to compare the $\mathcal{H}$-hierarchy with the
$\G_\xi$-hierarchy, for $\vec{\F}$ a regular chain of reductions first
define\footnote{We must define the $\omega_1$-chain on the
  $\vec{\F}$-degrees (rather than on sets) because the Perfect Set
  Property, which already follows from $\SLOL$, forbids the
  possibility of having an $\omega_1$-chain of sets of bounded Borel
  rank.} 
$s_\xi^\alpha [B]_{\vec{\F}}$  
(for $B \subseteq \RR$ and $1 \leq \alpha < \omega_1$) by letting
$s_\xi^1 [B]_{\vec{\F}} = [B]_{\vec{\F}}$,
$s_\xi^\alpha[B]_{\vec{\F}} = [\bigoplus_n C_n]_{\vec{\F}}$ (where
$C_n \in s_\xi^{\alpha_n}[B]_{\vec{\F}}$ and the $\alpha_n$'s are
increasing and cofinal in $\alpha$) if $\alpha$ is limit, and
$s_\xi^\alpha[B]_{\vec{\F}} = [s_{\vec{\F}}(C)]_{\vec{\F}}$ (where $C
\in s_\xi^{\alpha'}[B]_{\vec{\F}}$) if $\alpha = \alpha'+1$. (Note
that if $B \leq_{\vec{\F}} \neg B$ then the
$s^\alpha_\xi[B]_{\vec{\F}}$'s are exactly the $\omega_1$-chain of
consective $\vec{\F}$-selfdual degrees which follows
$[B]_{\vec{\F}}$.) Moreover, given a pointclass $\Gamma$ and a nonzero
ordinal $\mu < \omega_1$, define  
\[ {\rm PU}_\mu (\Gamma) = \left\{ \bigcup\nolimits_n (A_n \cap D_n)
  \mid A_n \in \Gamma \text{ and }\seq{D_n}{n \in \omega} \text{ is a
  }\bDelta^0_\mu\text{-partition of }\RR \right\}\] 
and
\[ {\rm SU}_\mu (\Gamma) = \left\{ \bigcup\nolimits_n (A_n \cap D_n)
  \mid A_n \in \Gamma,  D_n \in \bDelta^0_\mu \text{ and }D_n \cap D_m
  = \emptyset \text{ if } n \neq m \right\}.\] 
Finally, for $\mu$ limit and $\mu_n$'s (strictly) increasing and
cofinal in $\mu$ define ${\rm SU}_{< \mu, \alpha}(\Gamma)$ by the
following induction on $\alpha < \omega_1$ (note the definition is
independent from the choice of the $\mu_n$'s): 
\[ {\rm SU}_{< \mu,\alpha}(\Gamma) = 
\begin{cases}
\Gamma & \text{if }\alpha = 0\\
\bigcup_n {\rm SU}_{\mu_n} (\bigcup_{\alpha' < \alpha} {\rm
  SU}_{<\mu,\alpha'}(\Gamma)) & \text{if }\alpha>0. 
\end{cases}
\]

\begin{proposition}[$\ZF+\ACOR$]\label{propIIIvsII}
 Let $\vec{\F}$ be a regular chain of reductions of rank $\xi$ (for $\xi$ a countable limit ordinal). For $A,B \subseteq \RR$, $A \leq_{\sf{D}^\W_\xi} B$ if and only if $A \leq_\W C$ for some $C \in s_\xi^\alpha[B]_{\vec{\F}}$ and $\alpha < \omega_1$.
\end{proposition}

\begin{proof}
 Let $\bGamma = \bGamma(B) = \{  D \subseteq \RR \mid D \leq_\W B\}$ be the boldface pointclass generated by $B$. It is immediate to check that $A \leq_{\sf{D}^\W_\xi} B \iff A \in {\rm PU}_\xi(\bGamma)$. By Theorem E.4 of chapter IV of Wadge's \cite{wadgethesis}, ${\rm PU}_\xi(\bGamma) = \bigcup_{\alpha<\omega_1} {\rm SU}_{< \xi,\alpha} (\bGamma)$, so let $\alpha$ be smallest such that $A \in {\rm SU}_{< \xi, \alpha}(\bGamma)$. We will prove by induction on $\alpha$ that $A \leq_\W C$ for some $C \in s_\xi^{\alpha+1}[B]_{\vec{\F}}$. If $\alpha = 0$, then $A \in {\rm SU}_{<\xi,0}(\bGamma) = \bGamma$ and therefore $A \leq_\W B$ (by definition of $\bGamma$) and obviously $B \in [B]_{\vec{\F}} = s_\xi^1[B]_{\vec{\F}}$. Now assume $\alpha>0$, and let $n$ be such that $A \in {\rm SU}_{\mu_n} (\bigcup_{\alpha' < \alpha} {\rm SU}_{<\mu,\alpha'}(\bGamma))$, so that $A = \bigcup_m (A_m \cap D_m)$ where $D_m \in \bDelta^0_{\mu_n}$, $D_m \cap D_{m'} = \emptyset$ if $m \neq m'$, and $A_m \in {\rm SU}_{<\mu,\alpha_m}(\bGamma)$ for some $\alpha_m < \alpha$ (depending on $m$). By inductive hypothesis, $A_m \leq_\W C_m$ for some $C_m \in s_\xi^{\alpha_{m}+1}[B]_{\vec{\F}}$, and therefore $A_m \leq_\W \bigoplus_m C_m$ for every $m$. Moreover $\bigoplus_m C_m \in s_\xi^\alpha[B]_{\vec{\F}}$ as $\alpha = \sup \{ \alpha_m+1 \mid m \in \omega \}$ by its minimality. 

\begin{claim}
$A \leq_\W \Sigma^{\mu_n+1}(\bigoplus_m C_m)$.
\end{claim}

\begin{proof}[Proof of Claim]
Let $\seq{\mu_n}{n \in \omega}$ be the type of $\vec{\F}$, 
$P$ be the complete $\bPi^0_{\mu_n}$-set used to define the operator $\Sigma^{\mu_n+1}$, and $f_m$ be continuous functions such that $f_0$ reduces $\neg \bigcup_m D_m$ to $P$, $f_{2(m+1)}$ reduces $D_m$ to $P$, $f_1$ is constant with value $y \notin \bigoplus_m C_m$, and $f_{2m+3}$ is a reduction of $A_m$ to $\bigoplus_m C_m$: it is easy to check that $\bigotimes_m f_m$ reduces $A$ to $\Sigma^{\mu_n+1}(\bigoplus_m C_m)$ as required.
\renewcommand{\qedsymbol}{$\square$ \textit{Claim}}
\end{proof}

Since $\mu_n < \mu_n +1 \leq \mu_{n+1}$, we get $\Sigma^{\mu_n+1}(\bigoplus_m C_m) \leq_\W \Sigma^{\mu_{n+1}}(\bigoplus_m C_m) \leq_\W s_{\vec{\F}}(\bigoplus_m C_m) \in s_\xi^{\alpha+1}[B]_{\vec{\F}}$ and hence we are done.
\end{proof}

As a corollary of Proposition \ref{propIIIvsII} one gets a Steel--Van Wesep-style theorem for higher levels.

\begin{theorem}
Let $\G_\xi$ be as above and $\vec{\F}$ be a regular chain of reductions of rank $\xi$. Then $A \leq_{\G_\xi} \neg A$ if and only if $A \leq_{\vec{\F}} \neg A$.
In particular, $A \leq_{\sf{D}_\xi} \neg A$ implies that $A \leq_{\sf{D}_\mu} \neg A$ for some $\mu < \xi$.
\end{theorem}

\begin{proof}
One direction is easy, as $\bigcup_{\mu<\xi} \sf{D}_\mu \subseteq \sf{D}_\xi \simeq \G_\xi$. For the other direction let $B$ be $\L$-minimal in $[A]_{\G_\xi}$, so that $B \leq_\L \neg B$. As $\G_\xi \simeq \sf{D}^\W_\xi$, apply Proposition \ref{propIIIvsII} and  let $\alpha$ be minimal such that $A \leq_\W C$ for some $C \in s_\xi^\alpha[B]_{\vec{\F}}$. Since $A <_{\vec{\F}} C$ contradicts the minimality of $\alpha$, we must have $C \leq_{\vec{\F}} A$ and therefore $A \equiv_{\vec{\F}} C$: but as $C$ is $\vec{\F}$-selfdual (since $B$ is) we get $A \leq_{\vec{\F}} \neg A$ as desired.
\end{proof}

All this discussion solves the problem of comparing the $\mathcal{H}$-hierarchy with the $\G_\xi$-hierarchy:  using the fact that $\G_\xi \simeq \sf{D}^\W_\xi$, $\mathcal{H} \simeq \vec{\F}$ (where $\vec{\F}$ is any regular chain of rank $\xi$), and $s_{\vec{\F}} (A) \leq_{\sf{D}^\W_\xi} A$, we get that each $\G_\xi$-selfdual degree is exactly the union of a (maximal) $\omega_1$-chain of consecutive $\mathcal{H}$-selfdual degrees.

\section*{Appendix: Some alternative proofs of the {\bf SDP}}

In many of the concrete examples, one can directly prove (in a simpler way) that a Borel set of reductions $\F$ has the \textbf{SDP} using the fact that $\Lip(2) \subseteq \F$ and applying Proposition \ref{propequivSDP}. 
For instance, if $\F = \Lip$ or $\F = \sf{UCont}$, given a set $A \leq_\F \neg A$ which is $\L$-minimal in its $\F$-degree we can use the fact that $A \leq_\W \neg A$ (since $\Lip \subseteq \sf{UCont} \subseteq \W$) and then apply the Steel--Van Wesep Theorem to get $A \leq_\L \neg A$.

In the case of a regular chains of reductions $ \vec{\F} = \seq{\F_n}{n \in \omega}$, we can prove that if $A \leq_{\vec{\F}} \neg A$ and $B$ is $\L$-minimal in $[A]_{\vec{\F}}$ then $B \leq_\L
\neg B$ as follows:
let $n_0$ be minimal such that $A \leq_{\F_{n_0}} \neg A$, so that $A
\leq_{\F_m} \neg A$ for every $m \geq n_0$. Moreover, for every $m\geq n_0$
let $B_m$ be
$\L$-minimal in $[A]_{\F_m}$ (so that $B_m \leq_\L \neg B_m$ by Corollary
5.4 in \cite{mottorosborelamenability}), and
note that if $0 \leq k \leq m$ then $B_m
\leq_\L C$ for every $C \in [A]_{\F_k}$ because $\vec{\F}$ is regular
(in particular, if $n_0 \leq k
\leq m$ then $B_m \leq_\L B_k$). Since $\leq_\L$ is well-founded, there must be
some $n_1 \geq n_0$ such that $B_m \equiv_\L B_{n_1}$ for every $m \geq n_1$.
Put $B' = B_{n_1}$: then $B' \leq_\L \neg B'$  and $B'$ is easily seen to be $\L$-minimal in $[A]_{\vec{\F}}$, so that $B' \equiv_\L B$ and we are done.

However, Baire reductions $\sf{B}_\xi$ are perhaps the most interesting case.

\begin{lemmanonumber}[$\ZF+\ACOR$]
  Let $A,B \subseteq \RR$ be such that $A \leq_{\sf{B}_\xi} B$. Then
  there is some $C \equiv_{\sf{B}_\xi} A$ such that $C \leq_\L A$ and
  $C \leq_\L B$.
\end{lemmanonumber}

\begin{proof}
  It is enough to prove that there is some $A' \equiv_{\sf{B}_\xi} A$ such
  that $A' \leq_\W A$ and $A' \leq_\W B$: then applying twice Lemma 19 of \cite{andrettaequivalence} we get  the desired $C$ as in the proof of Lemma 8 in \cite{andrettamartin}.
  Let $\mu < \xi$ and $f \in
  \B_\mu$ be such that $A = f^{-1}(B)$. Applying Lemma
  \ref{lemmarel} to the family $\{f^{-1}(\bN_s)\mid s \in
  {}^{<\omega}\omega\}\subseteq \bDelta^0_{\mu+1}$, we get a new
  zero-dimensional Polish topology $\tau'\supseteq \tau$ on $\RR$ such
  that $f\colon  (\RR,\tau') \to (\RR,\tau)$ is continuous and
  $\bSigma^0_1(\tau') \subseteq \bSigma^0_{\mu+1}(\tau)$. Let $H\colon 
  (F,\tau) \to (\RR,\tau')$ be an homeomorphism between a closed set
  $F\subseteq \RR$ and $\RR$ endowed with the new topology, and let
  $r\colon \RR \onto F$ be a retraction on $F$. Finally, put $A' = (H \circ
  r)^{-1}(A)$: since  $H^{-1}\colon  (\RR,\tau) \to (\RR,\tau)$ is of Baire class $\mu$ we get  $A \leq_{\sf{B}_\xi} A'$,  and since
   $H \circ r$ and $f \circ H \circ r$ are clearly
  continuous function from $(\RR,\tau)$ to $(\RR,\tau)$ which witness $A' \leq_\W A$ and $A' \leq_\W B$,
  respectively, $A'$ is as required.
\end{proof}

Let now $B$ be $\L$-minimal in $[A]_{\sf{B}_\xi}$, where $A \leq_{\sf{B}_\xi} \neg A$. Since $B
  \leq_{\sf{B}_\xi} \neg B$, we can apply the previous lemma
  to get $C \equiv_{\sf{B}_\xi} B$
  such that $C \leq_\L B$ and $C \leq_\L \neg B$. By minimality of
  $B$, we must have $B \equiv_\L C$ and hence $B \leq_\L \neg B$.

\end{document}